\documentclass{article}
\usepackage[utf8]{inputenc}

\usepackage{amsmath}
\usepackage{amsfonts}
\usepackage{amssymb}
\usepackage{amsthm}
\usepackage{tikz}
\usepackage{booktabs}
\usepackage{subcaption}
\usepackage[capitalise]{cleveref}
\usepackage{algorithm}
\usepackage[noend]{algpseudocode}

\usepackage{pgfkeys} 	
\usepackage{ifthen} 		

\pgfkeys{
/tensor/.cd, 
dim1/.initial = 1,		dim1/.get = \dimOne,		dim1/.store in = \dimOne,
dim2/.initial = 1,		dim2/.get = \dimTwo,		dim2/.store in = \dimTwo,
dim3/.initial = 1,		dim3/.get = \dimThree,	dim3/.store in = \dimThree,
xshift/.initial = 0, 	xshift/.get = \xShift, 		xshift/.store in = \xShift,
yshift/.initial = 0, 	yshift/.get = \yShift, 		yshift/.store in = \yShift,
xspec/.initial = 0, 	xspec/.get = \xSpec, 		xspec/.store in = \xSpec,
yspec/.initial = 0, 	yspec/.get = \ySpec, 		yspec/.store in = \ySpec,
scale/.initial = 1, 	scale/.get = \myScale, 	scale/.store in = \myScale,
fill/.initial = white, 	fill/.get = \myFill, 		fill/.store in = \myFill,
back edges/.initial = 0, 	back edges/.get = \myBack, 	back edges/.store in = \myBack,
slice type/.initial = none, 		slice type/.get = \sliceType, 			slice type/.store in = \sliceType, 
number of slices/.initial = 1, 	number of slices/.get = \nSlices, 		number of slices/.store in = \nSlices,
slice width/.initial = 1, 		slice width/.get = \sWidth, 				slice width/.store in = \sWidth, 
}

\newcommand{\tensor}[1][]{\@tensor[#1]}
\def\@tensor[#1] (#2,#3) #4; {{ 

\pgfkeys{/tensor/.cd,#1}

\def\depthScale{0.5} 

\pgfmathsetmacro{\numSlicesMinusOne}{\nSlices-1}
\pgfmathsetmacro{\numSlicesPlusOne}{\nSlices+1}

\pgfmathsetmacro{\sliceLength}{\myScale*\dimOne}

\ifthenelse{\equal{\sliceType}{lateral}}
	{
	
	\pgfmathsetmacro{\sliceWidth}{\myScale*\sWidth*0.9*\dimTwo/\nSlices}
	\pgfmathsetmacro{\sliceGap}{\myScale*\dimTwo/(\nSlices-1) - \nSlices*\sliceWidth/(\nSlices-1)}
	\pgfmathsetmacro{\sliceDepth}{\myScale*\dimThree}
	
	} 
	{
	\ifthenelse{\equal{\sliceType}{frontal}}
		{
		
		\pgfmathsetmacro{\sliceDepth}{\myScale*\sWidth*0.9*\dimThree/\nSlices}
		\pgfmathsetmacro{\sliceGap}{\myScale*\dimThree/(\nSlices-1) - \nSlices*\sliceDepth/(\nSlices-1)}
		\pgfmathsetmacro{\sliceWidth}{\myScale*\dimTwo}
	
		}
		{
		\pgfmathsetmacro{\sliceWidth}{\myScale*\dimTwo}
		\pgfmathsetmacro{\sliceDepth}{\myScale*\dimThree}
		}

	}

\def\xFront{#2 + \xShift}	
\def\yFront{#3 + \yShift}
\def\xBack{#2 + \xShift + \depthScale*\sliceDepth + \xSpec*\sliceDepth}
\def\yBack{#3 + \yShift + \depthScale*\sliceDepth + \ySpec*\sliceDepth}

\def\aFront{(\xFront, \yFront)}
\def\bFront{(\xFront, \yFront + \sliceLength)}
\def\cFront{(\xFront + \sliceWidth, \yFront + \sliceLength)}
\def\dFront{(\xFront + \sliceWidth, \yFront)}

\def\aBack{(\xBack, \yBack)}
\def\bBack{(\xBack, \yBack + \sliceLength)}
\def\cBack{(\xBack + \sliceWidth, \yBack + \sliceLength)}
\def\dBack{(\xBack+ \sliceWidth, \yBack)}

\ifthenelse{\NOT\equal{\myFill}{nofill}}
	{
	\def\tempTensor{
		\fill[\myFill!25] \bFront -- \bBack -- \cBack -- \cFront -- cycle; 
		\fill[\myFill!75] \dFront -- \dBack -- \cBack -- \cFront -- cycle; 
		\fill[\myFill!50] \aFront rectangle \cFront;  
	
		\draw \aFront rectangle \cFront; 
		\draw \bFront -- \bBack; 
		\draw \cFront -- \cBack;
		\draw \dFront -- \dBack;
	
		\draw \bBack -- \cBack;
		\draw \cBack -- \dBack;
		}
	}
	{ 

	\def\tempTensor{
		\draw \aFront rectangle \cFront; 
		
		\ifthenelse{\NOT\equal{\myBack}{0}}
		{
			\draw[dashed] \bBack -- \aBack -- \dBack;
		}{}
		
		\draw \dBack -- \cBack -- \bBack;

		\ifthenelse{\NOT\equal{\myBack}{0}}
		{
			\draw[dashed] \aFront -- \aBack;
		}{}
		
		\draw \bFront -- \bBack;
		\draw \cFront -- \cBack;
		\draw \dFront -- \dBack;
		}
	}

\ifthenelse{\equal{\sliceType}{lateral}}
	{
	\foreach\sliceCount in {0,...,\numSlicesMinusOne}
		{	
		\begin{scope}[shift ={(\sliceCount*\sliceWidth + \sliceCount*\sliceGap, 0)}]
			\tempTensor;
		\end{scope}
		}
	
	}
	{
	
	\ifthenelse{\equal{\sliceType}{frontal}}
	{
	
	\pgfmathsetmacro{\xStep}{\sliceDepth/2 + \sliceGap/2 + \myScale*\dimThree*\xSpec/(\nSlices-(1-\sWidth))}
	\pgfmathsetmacro{\yStep}{\sliceDepth/2 + \sliceGap/2 +  \myScale*\dimThree*\ySpec/(\nSlices-(1-\sWidth))}
	
	\foreach\sliceCount in {-\numSlicesMinusOne,...,0}
		{	
		
		\begin{scope}[shift = {(-\sliceCount*\xStep, -\sliceCount*\yStep)}]
			\tempTensor;
		\end{scope}
	
		}
	
	}
	{
	\tempTensor;
	}
	
	}

\node at (#2 + \dimTwo/2, #3 + \dimOne/2) {#4};

}} 

\crefrangeformat{equation}{eqs.~(#3#1#4) to~(#5#2#6)}
\Crefrangeformat{fact}{Facts~#3#1#4-#5#2#6}

\renewcommand{\pmod}[1]{\ (\mathrm{mod}\ #1)}

\newcommand{\fold}[1]{\text{fold}\left(#1\right)}
\newcommand{\unfold}[1]{\text{unfold}\left(#1\right)}
\newcommand{\circulant}{\text{circ}}
\newcommand{\bcirc}[1]{\text{bcirc}\left(#1\right)}
\renewcommand{\circulant}[1]{\text{circ}\left(#1\right)}

\newcommand{\diag}[1]{\text{diag}\left(#1\right)}
\newcommand{\bdiag}[1]{\text{bdiag}\left(#1\right)}
\newcommand{\norm}[1]{\left\lVert#1\right\rVert} 
\newcommand{\E}[1]{\mathbb{E}\left[#1\right]} 
\newcommand{\C}{\mathbb{C}} 

\newcommand{\cA}{{\cal A}}
\newcommand{\cB}{{\cal B}}

\newcommand{\cE}{{\cal E}}

\newcommand{\cI}{{\cal I}}

\newcommand{\cM}{{\cal M}}

\newcommand{\cP}{{\cal P}}
\newcommand{\cQ}{{\cal Q}}

\newcommand{\cX}{{\cal X}}

\newcommand{\cW}{{\cal W}}

\newcommand{\mA}{{\bf A}}
\newcommand{\mB}{{\bf B}}

\newcommand{\mD}{{\bf D}}

\newcommand{\mF}{{\bf F}}

\newcommand{\mI}{{\bf I}}

\newcommand{\mM}{{\bf M}}

\newcommand{\mP}{{\bf P}}

\newcommand{\mX}{{\bf X}}
\newcommand{\mY}{{\bf Y}}

\newcommand{\thup}{^{\text{th}}}
\newcommand{\stup}{^{\text{st}}}

\theoremstyle{plain}
\newtheorem{theorem}{Theorem}  
\newtheorem{lemma}[theorem]{Lemma}
\newtheorem{fact}[theorem]{Fact} 
\Crefname{fact}{Fact}{Facts}

\theoremstyle{definition}
\newtheorem{assumption}{Assumption} 
\newtheorem{definition}{Definition}

\theoremstyle{remark}
\newtheorem{remark}{Remark} 
\Crefname{remark}{Remark}{Remarks}
\Crefname{assumption}{Assumption}{Assumptions}

\usepackage[colorinlistoftodos,bordercolor=orange,backgroundcolor=orange!20,linecolor=orange,textsize=scriptsize]{todonotes}

\usepackage{scalerel,stackengine}
\stackMath
\newcommand\reallywidehat[1]{%
\savestack{\tmpbox}{\stretchto{%
  \scaleto{%
    \scalerel*[\widthof{\ensuremath{#1}}]{\kern-.6pt\bigwedge\kern-.6pt}%
    {\rule[-\textheight/2]{1ex}{\textheight}}
  }{\textheight}%
}{0.5ex}}%
\stackon[1pt]{#1}{\tmpbox}%
}

\title{Randomized Kaczmarz for Tensor Linear Systems}
\author{Anna Ma and Denali Molitor}

\begin{document}

\maketitle

\begin{abstract}
Solving linear systems of equations is a fundamental problem in mathematics. When the linear system is so large that it cannot be loaded into memory at once, iterative methods such as the randomized Kaczmarz method excel. Here, we extend the randomized Kaczmarz method to solve multi-linear (tensor) systems under the tensor-tensor t-product. We provide convergence guarantees for the proposed tensor randomized Kaczmarz that are analogous to those of the randomized Kaczmarz method for matrix linear systems. We demonstrate experimentally that the tensor randomized Kaczmarz method converges faster than traditional randomized Kaczmarz applied to a naively matricized version of the linear system. In addition, we draw connections between the proposed algorithm and a previously known extension of the randomized Kaczmarz algorithm for matrix linear systems.
\end{abstract}

\section{Introduction}
Methods for processing and analyzing large datasets have seen rapid development and use in signal processing and machine learning. For example, in the machine learning community, recommender systems and collaborative filtering have become ubiquitous tools for understanding user behavior and preferences. Data is commonly interpreted in this setting as a user-item matrix. As another example, consider the process of recovering a compressed video. Videos are understood to be a collection of image frames and images are often vectorized so that the signal associated with a video is a pixel location by frame matrix.

One reason data are often organized in this two dimensional (user-item, pixel-frame, etc.) fashion is because a vast majority of the existing methods operate on data that are stored as matrices and vectors. Recommender systems employ matrix factorization~\cite{koren2009matrix}. Sparse optimization and multiple measurement vector methods are common approaches for video recovery~\cite{majumdar2012face, liu2010block}. In such approaches, optimization frameworks expect data in the form of one or two dimensional arrays (i.e., vectors and matrices). However, in reality, data can be higher multidimensional arrays and this restriction to the one or two dimensional representations often destroys structure (for example, spatial or temporal structure) inherent to the data. In video recovery, data occurs naturally as a third-order tensor with dimensions image width by image height by frame number. Commonly, images are vectorized to form columns of the pixel by frame data matrix, which destroys the spatial correlation within the frames. 

In the seminal paper of~\cite{kilmer2011factorization}, the authors define a closed multiplication operation between two tensors referred to as the \emph{t-product}. Initially motivated for tensor factorization, use of the t-product has become prominent in the tensor and signal processing community. Under the t-product, tensors enjoy a linear algebraic-like framework that has proved useful in applications such as dictionary learning \cite{soltani2016tensor, zhang2016denoising}, low-rank tensor completion \cite{8066348, 6737273, 7782758, 6909886}, facial recognition \cite{hao2013facial}, and neural networks \cite{newman2018stable, wang2020tensor}. 
The process of \emph{naively} transforming high-order tensors into two dimensional arrays via a flattening or unfolding process is often referred to as ``matricization". 
Since the t-product acts as a linear operator directly on higher-order tensors, it avoids matricization and preserves multidimensional structure.

Here, we consider the fundamental problem of solving large linear systems of equations for third-order tensors under the t-product. In the matrix linear system setting, randomized iterative methods are a popular choice for solving or finding approximate solutions to systems that are too large to load into memory at once \cite{drineas2016randnla, ma2015convergence, strohmer2009randomized}. 

One such randomized iterative method is the known as the randomized Kaczmarz method. The randomized Kaczmarz method (MRK) \footnote{While the randomized Kaczmarz literature typically abbreviates randomized Kaczmarz as RK, throughout this work, MRK is used to distinguish the matrix and tensor versions of randomized Kaczmarz.} is closely related to other popular randomized iterative methods such as stochastic gradient descent and coordinate descent and is commonly used in computed tomography (CT imaging) and other signal processing applications~\cite{needell2014stochastic, gower2015randomized}. 

In this work, we propose a Kaczmarz-type iterative methods for tensor linear systems under the t-product which we refer to as tensor randomized Kaczmarz (TRK). We analyze the convergence of TRK and derive theoretical guarantees for the proposed method in two variations. The first approach analyzes TRK from a similar lens as MRK, i.e., views iterates as projections onto solution spaces of a subsampled system. The second approach takes advantage of the fact that the t-product can be efficiently computed in Fourier space. In addition to proving theoretical guarantees for TRK, we also make connections between TRK and other variants of MRK. Our theoretical findings are supported by numerical experiments before we conclude our work with final remarks. 
We view this analysis as a case study with a template to extend other methods to the tensor setting under the t-product.

\subsection{Randomized Kaczmarz}
Randomized Kaczmarz is an iterative method for approximating solutions to linear systems of equations \cite{Kaczmarz1937}. The MRK method uses iterative projections onto the solution space with respect to a selected row to approximate the solution of a linear system. More specifically, for a linear system $\mA x=b$, a row index $i$ is chosen at each iteration of MRK and the current iterate (approximate solution) is projected onto the solution space
\begin{equation*}
    \mA_{i:} x = b_i.
\end{equation*}
The method is advantageous for very large linear systems that cannot be loaded into memory at once. 
There are many extensions to MRK including greedy \cite{agmon1954relaxation,motzkin1954relaxation,haddock2019motzkin,de2017sampling,nutini2016convergence,petra2016single,censor1981row} and block \cite{needell2014paved} variants to speed convergence and an extended version for inconsistent linear systems \cite{needell2010randomized,zouzias2013randomized}. 
The MRK method and its block variant fall under the more general sketch-and-project framework which additionally includes other popular methods such as coordinate descent \cite{gower2015randomized}. 
Strohmer and Vershynin demonstrated that MRK converges exponentially in expectation when indices $i$ are sampled with probabilities proportional to the squared row norms $||\mA_{i:}||^2$ \cite{strohmer2009randomized}. When the rows of $\mA$ are normalized, this is equivalent to sampling the indices uniformly at random. The standard MRK update for a linear system $\mA x = b$ is given by 
\begin{equation}\label{eqn:matrixRK}
    x^{t+1} = x^{t} 
     - \mA_{i_t:}^* \frac{ \langle \mA_{i_t:}, x^t\rangle - b_{i_t}}{\norm{\mA_{i_t:}}^2},
\end{equation}
where $i_t$ is the row index selected at iteration $t$ and $\mA_{i_t:}^*$ is the transpose of the $i_t\thup$ row of $\mA$.
At each iteration $t$, the current iterate $x^t$ is projected onto the solution space with respect to the row $\mA_{i_t:}$ of the measurement matrix $\mA$.

For linear systems of the form $\mA\mX = \mB$ with $\mX$ and $\mB$ representing matrices, we can apply the MRK update of \Cref{eqn:matrixRK} to each column of $\mB$ in order to recover each column of the signal $\mX$. We can equivalently rewrite the MRK update for this case as 
\begin{equation}\label{eqn:matrix_simul_RK}
    \mX^{t+1} = \mX^{t} -\mA_{i_t:}^* \frac{ \mA_{i_t:} \mX^t - \mB_{i_t:}}{\norm{\mA_{i_t:}}^2}.
\end{equation}

\subsection{Tensor linear systems}
Tensors arise in many applications and working with tensors directly, as opposed to naively flattening tensors into matrices can preserve significant structures and  have computational advantages. 
Unfortunately, when working with tensors, many basic and fundamental linear algebraic constructs and results do not generalize naturally. For example, it is not obvious how one should define multiplication between two tensors \cite{kilmer2011factorization,bader2006algorithm,kiers2000towards,de2000multilinear}.

We specifically consider tensor linear systems under the tensor t-product. The t-product, proposed by Kilmer and Martin \cite{kilmer2011factorization}, is a bilinear operation between tensors, that allows for the generalization of many matrix algebra definitions and properties to the tensor setting. In particular, the t-product generalizes the concept of orthogonality between tensors, which is key for analysis of TRK. We provide further details about the t-product in \Cref{subsec:tensor_LA}.

A tensor linear system under the t-product is formulated as follows. Let $\cX\in\C^{\ell \times p\times n}$ be an unknown third-order tensor representing a  three-dimensional data array. For example, this three-dimensional data could represent a video, color image, temporal data, or three-dimensional density values. 
A tensor linear system under the t-product is written as:
\begin{equation}\label{eqn:linsys}
    \cA \cX = \cB,
\end{equation} 
with $\cA \in \C^{m \times \ell \times n}$, $\cX \in \C^{\ell \times p \times n}$ and $\cB\in \C^{m\times p\times n}$. 

Tensor linear systems arise in many applications. For example, factorization methods and dictionary learning have been extended to the tensor setting \cite{Zubair2013,Zhang2017,tan2015tensor,Roemer2014,anandkumar2015learning} and specifically with use of the t-product \cite{kilmer2011factorization,2019KilmerNewton,soltani2016tensor}. In practice, factorization methods such as non-negative matrix factorization depend on solving (potentially very large) linear systems as subroutines. As another example, consider extreme learning machines (ELM). Extreme learning machines are feedforward neural networks in which random weights are assigned for the hidden nodes \cite{huang2006universal}. A linear mapping of the hidden-layer outputs is then learned using a labeled set of training data. A major advantage of ELM is that learning the linear mapping of the hidden-layer outputs to the output layer is relatively simple and is independent of the activation functions used. Newman et al. proposed a tensor neural network intended for tensor data \cite{newman2018stable}. Their proposed networks use tensor-tensor products and enable the use of more compact parameter spaces \cite{2019KilmerNewton}. Extending ELM to the tensor neural network setting leads to the need to solve (again potentially very large) tensor linear systems.

Randomized Kaczmarz is closely related to the popular optimization technique, stochastic gradient descent (SGD) \cite{needell2014stochastic}. Most related to this work is the 
tensor stochastic gradient descent that was recently implemented to train tensor neural networks under the t-product~\cite{newman2018stable}. The focus of the aforementioned work is a tensor neural network framework for multidimensional data and does not delve into an algorithmic analysis of SGD under the t-product. 

In this work, we introduce a Kaczmarz inspired iterative method that considers row slices of the tensor system at each iteration and provide theoretical analysis for the proposed method. To the best of our knowledge, no other works have consider solving large-scale linear t-product tensor systems with stochastic iterative methods.

\subsection{Contributions}
We propose TRK, a randomized Kaczmarz method for solving linear systems of third-order tensors under the t-product. We analyze the convergence of the proposed method and demonstrate its performance empirically. In the convergence analysis, we discuss projections and spectral constants for tensors under the t-product.  We compare the performance of the proposed TRK method with a naively matricized MRK applied to a flattened tensor system. We also demonstrate that TRK is equivalent to performing block MRK \cite{elfving1980block} in the Fourier domain. This work serves as an example for extending methods for tensors under the t-product and how the properties of the t-product in the Fourier domain can be used to analyze convergence in this setting.

\section{Background and notation}\label{sec:background_notation}
In this section, we present notation and several linear algebraic results for tensors under the t-product.

\subsection{Notation}\label{subsec:notation}
Throughout, calligraphic capital letters represent tensors, bold capital letters represent matrices, and lower case letters represent vectors and scalars. The index $i$ is reserved for indexing row slices of tensors (see~\Cref{fig:row_slice}), rows of matrices, and entries of vectors. The index $j$ is similarly reserved for indexing column slices of tensors and columns of matrices. The index $k$ is reserved for indexing frontal slices of tensors as illustrated in \Cref{fig:frontal_slice}.

For matrices \mM, we use the notation $\mM_{i:}$ and $\mM_{:j}$ to represent the $i\thup$ row and $j\thup$ column respectively. We use $\cM_{i::}$ to represent row slices and $\cM_{::k}$ to represent frontal slices of a third-order tensor $\cM$ as shown in \Cref{fig:tensor_slices}. Because frontal slices of tensors are heavily used throughout this work, to condense notation,  bold subscripted capital letters, $\mM_k$, represents the $k\thup$ frontal slice of $\cM$ equivalently given by $\cM_{::k}$, unless otherwise stated (for example, the $n\times n$ DFT matrix $\mF_n$ and  $n\times n$ identity matrix $\mI_n$). 

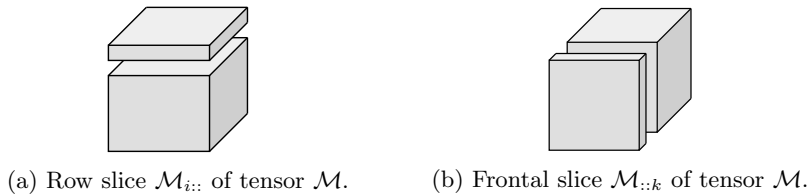
\begin{figure}[h]
    \begin{center}
\begin{tabular}{ccc}

\begin{subfigure}{.45\linewidth}
\centering
  \resizebox{0.4\textwidth}{!}{
    \begin{tikzpicture}
    	\tensor[dim1 = 0.75, dim2 = 1., dim3 = 0.75, fill = lightgray] (6  ,-.9) {};
        \tensor[dim1 = 0.15, dim2 = 1., dim3 = 0.75, fill = lightgray] (6 ,0) {};
    \end{tikzpicture}
  }
\caption{Row slice $\cM_{i::}$ of tensor $\cM$.}
\label{fig:row_slice}
\end{subfigure}
&

\begin{subfigure}{.45\linewidth}  
    \centering
  \resizebox{0.4\textwidth}{!}{
    \begin{tikzpicture}	
    	\tensor[dim1 = 1., dim2 = 1., dim3 = .75, fill = lightgray] (10  ,-.9) {};
        \tensor[dim1 = 1., dim2 = 1., dim3 = 0.15, fill = lightgray] (9.8 ,-1.1) {};
    \end{tikzpicture}
    }
    \caption{Frontal slice $\cM_{::k}$ of tensor $\cM$.}
    \label{fig:frontal_slice}
\end{subfigure}

\end{tabular}
\end{center}
\caption{Row slice $\cM_{i::}$ and frontal slice $\cM_{::k}$ of tensor $\cM$.}
\label{fig:tensor_slices}
\end{figure}

The squared Frobenius norm $||\cdot||_F^2$ for matrices and tensors denotes the sum of squares of all scalar elements. For a matrix $\mM$, $||\mM||_F^2 = \sum_{ij} \mM_{ij}^2$ and for a third-order tensor $\cM$, $||\cM||_F^2 = \sum_{ijk} \cM_{ijk}^2$. We use $\sigma_{\min}(\mM)$ to denote the smallest singular value and $\mM^\dagger$ to denote the pseudoinverse of the matrix $\mM$.

\Cref{eqn:unfold} shows how a third-order tensor $\cM$ is unfolded into a matrix: 
\begin{equation}\label{eqn:unfold}
    \unfold{\cM} = \begin{pmatrix}
         \cM_{::0}  \\
         \vdots \\
         \cM_{::n-1}
    \end{pmatrix}= \begin{pmatrix}
         \mM_{0}  \\
         \vdots \\
         \mM_{n-1}
    \end{pmatrix}.
\end{equation}
To revert the unfolding of a tensor $\cM$ we can fold the matrix in \Cref{eqn:unfold} such that $\fold{\unfold{\cM}} = \cM$. 
To condense notation, when using both indices and transposes, the transpose are applied to the tensor or matrix slice, that is $\mM_{i:}^* = \left(\mM_{i:}\right)^*$ and $\cM_{i::}^* = \left(\cM_{i::}\right)^*$. 

The tensor product of tensors $\cA$ and $\cB$ is written as $\cA\cB$. Similarly, for matrices $\mA, \mB$, their matrix product is written as $\mA\mB$. We do not consider the products between tensors and matrices. 
Throughout, we use $\cA$ and $\mA$ to represent the measurement tensor and matrix, $\cX$, $\mX$, and $x$ to represent signal tensor, matrix and vector and $\cB$, $\mB$, and $b$ to represent the observed measurements for the linear systems
\begin{equation*}
    \cA \cX = \cB, \quad \mA \mX = \mB, \mbox{ and } \mA x = b.
\end{equation*}
Lastly, the index $t$ is reserved only to indicate iteration number and the shorthand $i \in [m-1]$ denotes $i = \{ 0, 1, 2, ..., m-1\}$.

\subsection{Tensor linear algebra}\label{subsec:tensor_LA} 

We now provide background on the tensor-tensor t-product\cite{kilmer2011factorization}. Under the t-product one, can recover many standard linear algebraic properties such as transposes, orthogonality, inverses and projections.

The t-product is defined in terms of block-circulant matrices.
\begin{definition}\label{defn:bcirc}
For $\cA\in \C^{m \times \ell \times n}$, let $\bcirc{\cA}$ denote the block-circulant matrix
\begin{equation}\label{eqn:bcirc}
    \bcirc{\cA} = \begin{pmatrix}
        \mA_0 & \mA_{n-1} & \mA_{n-2} & \ldots & \mA_{1}\\
        \mA_{1} & \mA_{0} & \mA_{n-1} & \ldots & \mA_{2}\\
        \vdots & \vdots & \vdots & \ddots & \vdots \\
        \mA_{n-1} & \mA_{n-2} & \mA_{n-3} & \ldots & \mA_{0}\\
        \end{pmatrix} \in\C^{mn \times \ell n}.
\end{equation}
\end{definition}

The following definitions of tensor-tensor product, identity tensor, transpose, inverse, and orthogonality under the t-product are taken from Kilmer and Martin  \cite{kilmer2011factorization}. While the definitions and results here are specific to the t-product, this product has been generalized to a class of tensor products that use arbitrary invertible linear operators \cite{kernfeld2015tensor}. 
\begin{definition}\label{defn:tprod}
The tensor-tensor t-product is defined as 
\begin{equation*}
    \cA \cB = \fold {\bcirc{\cA} \unfold{\cB}} \in \C^{m \times p \times n},
\end{equation*}
where $\cA \in \C^{m \times \ell \times n}$ and $\cB \in \C^{\ell \times p \times n}$.
\end{definition}

\begin{definition}\label{defn:Identity}
The $m\times m\times n$ identity tensor, denoted $\cI$, is the tensor whose first frontal slice is the $m\times m$ identity matrix and whose remaining entries are all zeros.
\end{definition}
The identity tensor satisfies
\[\cM  \cI = \cI  \cM = \cM\] 
for all tensors $\cM$ with compatible sizes. 

\begin{definition}\label{defn:transpose}
The conjugate transpose of a tensor $\cM \in \C^{m\times \ell \times n}$ is denoted $\cM^*$ and is produced by taking the conjugate transpose of all frontal slices and reversing the order of the frontal slices $1,\ldots, n-1$.
\end{definition}
Note that this definition ensures $\left(\cM^*\right)^* = \cM$ and $(\cA \cB)^* = \cB^* \cA^*$. A tensor is symmetric if $\cM^*=\cM$.

\begin{definition}
    A tensor $\cM$ is invertible if there exists an inverse tensor $\cM^{-1}$ such that
    \begin{equation*}
        \cM \cM^{-1} = \cM^{-1} \cM = \cI.
    \end{equation*}
\end{definition}
Note that for an invertible tensor $\cM$, 
\begin{equation*}
    \cM^* \left(\cM^{-1}\right)^* = \cM^{-1} \cM = \cI
\end{equation*}
and 
\begin{equation*}
   \left(\cM^{-1}\right)^* \cM^*  =  \cM \cM^{-1} = \cI.
\end{equation*}
Thus, we have $\left(\cM^*\right)^{-1} = \left(\cM^{-1}\right)^* $.

\begin{definition}\label{defn:orthogonal}
A tensor $\cQ\in \C^{m\times p\times n}$ is orthogonal if 
\begin{equation*}
    \cQ^*  \cQ = \cI = \cQ  \cQ^*.
\end{equation*}
\end{definition}

The following properties of block circulant matrices will be useful throughout. Proofs can be found in \Cref{sec:bcirc_proofs}.
\begin{fact}
For tensors $\cA$ and $\cB$, the following equality holds:
    \begin{equation*}
        \bcirc{\cA\cB} = \bcirc{\cA}\bcirc{\cB}.
    \end{equation*}
    \label{fact:bcirc_distr}
\end{fact}

\begin{fact}\label{fact:bcirc_transpose}
The block circulant operator $\bcirc{\cdot}$ commutes with the conjugate transpose,
    \begin{equation*}
        \bcirc{\cM^*} = \bcirc{\cM}^*.
    \end{equation*}
\end{fact}

\subsubsection{Orthogonal tensor projections}\label{subsubsec:ortho_proj}
These definitions and facts allow us to characterize orthogonal tensor projections under the t-product, which is key for proving convergence of TRK.

\begin{lemma}\label{lem:projections}
If $\cM^* \cM$ is invertible, then the tensor $\cP =  \cM \left(\cM^* \cM \right)^{-1} \cM^*$ is an orthogonal projection tensor.
\footnote{This result is also stated and discussed in \cite{kilmer2013third}. We provide a proof here for completeness.}
\end{lemma}
\begin{proof}
First, we show that $\cP=  \cM \left(\cM^* \cM \right)^{-1} \cM^*$ is a projection tensor, which follows by the following computation:
        \begin{align*}
        \cP\cP &= 
        \cM \left(\cM^*\cM \right)^{-1} \cM^*
        \cM \left(\cM^*\cM \right)^{-1} \cM^*\\
        &=  \cM \left(\cM^*\cM \right)^{-1} \cM^*\\
        &=\cP.
    \end{align*}
    
    From the multiplication reversal property of the Hermitian transpose given in Proposition 4.3 of \cite{kernfeld2015tensor},
    \begin{align*}
        \cP^* 
        &= \left(\cM \left(\cM^*\cM \right)^{-1} \cM^* \right)^* \\
        &= \cM \left(\left(\cM^*\cM \right)^{-1}\right)^* \cM^*\\
        &= \cM \left(\cM^*\cM \right)^{-1}\cM^*\\
        &= \cP
    \end{align*}
    
    and the tensor $\cP$ is an orthogonal projection.
\end{proof}

The convergence analysis of TRK, uses the following result.
\begin{lemma}\label{lem:bcirc_proj}
    If the tensor $\cP\in \C^{m\times m\times n}$ is an orthogonal projection, $\bcirc{\cP}$ is also an orthogonal projection.
\end{lemma}
\begin{proof}
    Since $\cP$ is symmetric,
    \begin{equation*}
        \bcirc{\cP}^* \overset{\Cref{fact:bcirc_transpose}}{=} \bcirc{\cP^*} =\bcirc{\cP}.
    \end{equation*}
    
To see that $\bcirc{\cP}$ is a projection, note that since $\cP$ is a projection tensor, 
\begin{equation*}
    \bcirc{\cP } = \bcirc{\cP\cP} \overset{\Cref{fact:bcirc_distr}}{=} \bcirc{\cP} \bcirc{\cP}.
\end{equation*}

\end{proof}

\section{Tensor randomized Kaczmarz}\label{sec:TRK}

Tensor randomized Kaczmarz is a Kaczmarz-type iterative method designed for t-product tensor linear systems. One notable difference between the t-product tensor and matrix linear systems is the interaction of the measurements $\cA_{i::}$ and $\mA_{i:}$ with the signals $\cX$ and $x$. For the products $\mA_{i:} x = b_i$ and $\mA_{i:} \mX = \mB_{i:}$, each value in the signal $\mX$ or $x$ is multiplied by a single element of the measurement $\mA_{i:}$.
In the tensor measurement product,
\begin{equation*}
    \cA_{i::} \cX = \fold{\bcirc{\cA_{i::}}\unfold{\cX}}\in \C^{n\times p}.
\end{equation*}
Since $\bcirc{\cA_{i::}}\in \C^{n\times\ell n}$, each element of $\cX$ is multiplied by $n$ elements in $\cA_{i::}$ and affects $n$ entries of the resulting product $\cB_{i::}$. 
Equivalently, each frontal face of $\cX$ is multiplied by each frontal face of $\cA_{i::}$. See Kilmer and Martin \cite{kilmer2011factorization} for more details and intuition for the t-product.

We propose the following TRK update for tensor linear systems
\begin{equation}\label{eqn:trk_update}
    \cX^{t+1} = \cX^t - \cA_{i_t::}^* \left(\cA_{i_t::}\cA_{i_t::}^*\right)^{-1}\left(\cA_{i_t::} \cX^t-\cB_{i_t::}\right).
\end{equation}
The $\cA_{i::}$ are row slices of the tensor $\cA$ as depicted in \Cref{fig:row_slice}. The index $i_t$ used at each iteration is selected according to a probability distribution over the row indices $i\in[m-1]$. The TRK algorithm is detailed in \Cref{algo:TRK}.

\begin{algorithm}[t]

\begin{algorithmic}
\State \textbf{Input:}  $\cX^0\in \C^{\ell\times p\times n},$ $\cA\in \C^{m\times \ell \times n}$, $\cB\in\C^{m\times p \times n}$,  and probabilities $p_0, \dots, p_{m-1}$ corresponding to each row slice of $\cA$
\For {$t = 0, 1, 2, \dots$}
    \State  Sample $i_t \sim p_i$
	\State  $\cX^{t+1} = \cX^t - \cA_{i_t::}^* \left(\cA_{i_t::}\cA_{i_t::}^*\right)^{-1}\left(\cA_{i_t::} \cX^t-\cB_{i_t::}\right).$
\EndFor
\State \textbf{Output:} last iterate $\cX^{t+1}$
\end{algorithmic}
\caption{Tensor RK}
\label{algo:TRK}
\end{algorithm}

Let $\cP_{i} = \cA_{i::}^* \left(\cA_{i::}\cA_{i::}^*\right)^{-1}\cA_{i::}$. Under the assumption that $\cA_{i::}\cA_{i::}^*$ is invertible, by \Cref{lem:projections}, $\cP_{i}$ is an orthogonal projection onto the range of $\cA_{i::}$. Consequently, at each iteration, the current iterate $\cX^{t}$ is projected onto the solution space of the sub-sampled system $\cA_{i_t::}\cX = \cB_{i_t::}$. Note that this is the natural analogue of the MRK update, which projects the current iterate $x^t$ onto the solution space of $\mA_{i_t} x = b_{i_t}$.

Recall the MRK update given in \Cref{eqn:matrixRK}.
The multiplication by $\left(\cA_{i_t ::} \cA_{i_t::}^*\right)^{-1}$ in the TRK update serves an analogous role to normalization by the squared row norms, $\norm{\mA_{i_t:}}^2$, in the MRK update. The following assumption insures that the tensor $\cA_{i_t ::} \cA_{i_t::}^*$ is invertible so that the iterates are well defined.
\begin{assumption}\label{assump:invertibility_req}
Assume that $\cA_{i::}\cA_{i::}^*$ is invertible.
\end{assumption}
Note that in order for  $\cA_{i::}\cA_{i::}^*$ to be invertible, $ \diag{\mF_n \bcirc{\cA_{i::}\cA_{i::}^*} \mF_n^*} = \diag{\mD}$ must contain no non-zero entries. If the matrix $\bcirc{\cA_{i::}\cA_{i::}^*}$ is invertible, then the tensor $\cA_{i::}\cA_{i::}^*$ is also and its inverse can be calculated explicitly as follows.

\begin{lemma}\label{lem:inverse_AAT}
    The inverse of $\cA_{i::}\cA_{i::}^*$ under the t-product is 
\begin{equation}
\label{eq:invAAT}
\left(\cA_{i::}\cA_{i::}^*\right)^{-1} =\fold{\frac{1}{\sqrt{n}} \mF_n^* \; \diag{\mD^{-1}} },
    \end{equation}
    where $\mF_n$ is the $n\times n$ Discrete Fourier Transform (DFT) matrix and $\mD$ is a diagonal matrix such that $\bcirc{\cA_{i::}\cA_{i::}^*} = \mF_n^* \mD \mF_n$.
\end{lemma}

\begin{proof}
    The inverse of $\cA_{i::}\cA_{i::}^*$ is given by the tube fiber $\cW$ that satisfies 
    \begin{equation*}
        \cA_{i::}\cA_{i::}^* \cW = \fold {\bcirc{\cA_{i::}\cA_{i::}^*} \unfold{\cW})}  = \cI.
    \end{equation*}
    Since $\cA_{i::}\cA_{i::}^*\in \C^{1\times 1\times n}$ is a tube fiber, $\bcirc{\cA_{i::}\cA_{i::}^*}$ is a circulant matrix. Circulant matrices are diagonalizable by the DFT given by $\mF_n$, and we can thus write $\bcirc{\cA_{i::}\cA_{i::}^*} = \mF_n^* \mD \mF_n$ for some diagonal matrix $\mD$. Inverting this, we have $[\bcirc{\cA_{i::}\cA_{i::}^*}]^{-1} = \mF_n^* \mD^{-1} \mF_n$.
    
    Thus 
    \begin{align*}
        \unfold {\cW} &= \left[\bcirc{\cA_{i::}\cA_{i::}^*}\right]^{-1}\unfold{\cI} = \mF_n^* \mD^{-1} \mF_n \unfold{\cI}.
    \end{align*}
Using the definition of the DFT matrix, 
    \begin{equation*}
        \mF_n \unfold{\cI} 
        = \mF_n \begin{pmatrix} 
        1 \\
        0 \\
        \vdots \\
        0
        \end{pmatrix}
        = \frac{1}{\sqrt{n}}\begin{pmatrix} 
        1 \\
        1 \\
        \vdots \\
        1
        \end{pmatrix}.
    \end{equation*}
    Thus, 
    \begin{equation*}
        \cW = \fold{\frac{1}{\sqrt{n}} \mF_n^* \; \diag{\mD^{-1}}} .
    \end{equation*}
\end{proof}

\section{Convergence}
We demonstrate that the TRK method given by the update in \Cref{eqn:trk_update} satisfies a convergence result analogous to that of the matrix, vector setting. \Cref{thm:conv} shows that in expectation, the TRK algorithm will converge linearly to the solution of a consistent tensor system if 
\begin{equation*}
    \rho := 1-\sigma_{\min}(\E{\bcirc{\cP_{i}}}) < 1.
\end{equation*} 
The constant $\rho$  is often referred to as the \emph{contraction coefficient}. To show that this term is indeed less than one, we take advantage of the fact that the t-product can be computed in Fourier space. This analysis is presented in~\Cref{sec:Fourier_analysis}. In this section, the TRK algorithm is analyzed with a more classical approach for Kaczmarz-type algorithms and the result is compared to the standard MRK convergence guarantee.

\begin{theorem}\label{thm:conv}
Let $\cX^*$ be such that $\cA \cX^* = \cB$ and $\cX^t$ be the $t\thup$ approximation of $\cX^*$ given by the updates of \Cref{eqn:trk_update} with initial iterate $\cX^0$ and indices $i$ sampled independently from a probability distribution $\mathcal{D}$ at each iteration. Denote the orthogonal projection $\cP_{i} = \cA_{i::}^* \left(\cA_{i::}\cA_{i::}^*\right)^{-1}\cA_{i::}$. The expected error at the $t+1\stup$ iteration satisfies 
\begin{align*}
    \E{\norm{\cX^{t+1}-\cX^*}_F^2\middle| \cX^{0}} \le (1-\sigma_{\min}(\E{\bcirc{\cP_{i}}}))^{t+1} \norm{\cX^{0}-\cX^*}_F^2,
\end{align*} 
where the expectation is taken over the probability distribution $\mathcal{D}$, $\sigma_{\min}(\mM)$ denotes the smallest singular value of $\mM$, and $\|\cM\|_F^2$ is the sum of squared entries of the tensor $\cM$.
\end{theorem}
The proof of \Cref{thm:conv} mirrors the standard analysis of MRK making use of the linear algebra mimetic properties of the t-product. More specifically, the proof proceeds as follows. First, we show that the expected error at the $t\thup$ iteration is bounded above by the error from the previous iteration minus a projected error term using a tensor Pythagorean theorem. Then, a lower bound on the norm of the projected error is obtained to lead to the desired result. 
The proof of \Cref{thm:conv} is provided here and more technical components that extend simple properties for tensors are deferred to the appendix. 

\begin{proof}
Let $\cX^*$ be such that $\cA \cX^* = \cB$. 
Subtracting $\cX^*$ from both sides of the TRK update given in  \Cref{eqn:trk_update},
\begin{align*}
    \cX^{t+1}-\cX^* &= \cX^t-\cX^* - \cA_{i_t::}^* \left(\cA_{i_t::}\cA_{i_t::}^*\right)^{-1}\left(\cA_{i_t::} \cX^t-\cB_{i_t::}\right)\\
    &= \cX^t-\cX^* - \cA_{i_t::}^* \left(\cA_{i_t::}\cA_{i_t::}^*\right)^{-1}\cA_{i_t::}\left( \cX^t-\cX^*\right)\\
    &= \left( \cI - \cP_{i_t}\right)\left( \cX^t-\cX^*\right).
\end{align*}
To simplify and condense notation, we will use $\cE^{t} = \cX^{t}-\cX^*$ to represent the error at iteration $t$. Taking the Frobenius norm of the equality above,
\begin{equation*}
    \norm{\cE^{t+1}}^2_F = \norm{\left( \cI - \cP_{i_t}\right) \cE^t}_F^2.
\end{equation*}

As holds for orthogonal matrix projections, we can decompose this error as 
\begin{equation}\label{eqn:pythag_orthog}
      \norm{\left( \cI - \cP_{i_t}\right)\cE^t}_F^2 =\norm{\cE^t}_F^2 - \norm{ \cP_{i_t}\cE^t}_F^2
\end{equation}
using the Pythagorean theorem (see \Cref{lem:pythag}). 

Thus, \Cref{eqn:pythag_orthog} holds and
\begin{equation*}
        \norm{\cE^{t+1}}_F^2 = \norm{\cE^t}_F^2 - \norm{\cP_{i_t}\cE^t}_F^2.
\end{equation*}

Since the distribution from which the rows are sampled is fixed for all iterations, we drop the dependence on the iteration $t$ when taking expectations.
Taking the expectation over all row slice indices $i$, 
\begin{equation}\label{eqn:exp_outside}
        \E{\norm{\cE^{t+1}}_F^2\middle| \cX^{t}} = \norm{\cE^t}_F^2 - \E{\norm{\cP_i\cE^t}_F^2}.
\end{equation}
Note that
\begin{align*}
    \norm{\cP_i\cE^t}_F^2
    &=\norm{\bcirc{\cP_i}\unfold{\cE^t}}_F^2
     =\sum _{j=1}^{p}\norm{\bcirc{\cP_i}\unfold{\cE^t}_{:j}}_2^2.
\end{align*}

Now, since $\bcirc{\cP_{i}}$ is an orthogonal projection,
\begin{align}\label{eqn:innerprod}
    \E{\norm{\cP_i\cE^t}_F^2}
    &= \sum _{j=1}^{p}\E{\langle \bcirc{\cP_i}\unfold{\cE^t}_{:j},\bcirc{\cP_i}\unfold{\cE^t}_{:j}\rangle} \nonumber \\
    &= \sum _{j=1}^{p}\langle \E{\bcirc{\cP_i}}\unfold{\cE^t}_{:j},\unfold{\cE^t}_{:j}\rangle.
\end{align}
Since $\E{\bcirc{\cP_i}}$ is symmetric, 
\begin{align*}
    \E{\norm{\cP_i\cE^t}_F^2} 
    & \ge \sigma_{\min}\left(\E{\bcirc{\cP_i}}\right) \sum _{j=1}^{p} \norm{\unfold{\cE^t}_{:j}}_2^2\\
    & \ge \sigma_{\min}\left(\E{\bcirc{\cP_i}}\right) \sum _{j=1}^{p} \norm{\cE^t_{:j:}}_F^2\\
    & \ge \sigma_{\min}\left(\E{\bcirc{\cP_i}}\right)  \norm{\cE^t}_F^2.
\end{align*}
Making this substitution in \Cref{eqn:exp_outside}, we then have 
\begin{align*}
    \E{\norm{\cE^{t+1}}_F^2 \middle| \cX^{t}} \le (1-\sigma_{\min}(\E{\bcirc{\cP_{i}}})) \norm{\cE^t}_F^2.
\end{align*} 
Since the row slice indices $i$ are sampled independently, the conditional expectation can be iterated to obtain, 
\begin{align*}
    \E{\norm{\cE^{t+1}}_F^2\middle| \cX^{0}} \le (1-\sigma_{\min}(\E{\bcirc{\cP_{i}}}))^{t+1} \norm{\cE^0}_F^2.
\end{align*}
\end{proof}

The convergence guarantee of \Cref{thm:conv} is analogous to that of \cite{strohmer2009randomized} for MRK. 
If rows are sampled with probabilities proportional to the squared row norms of $\mA$, the expected approximation error for iterates of MRK is upper bounded as:
\begin{align*}
    \E{\norm{x^{t+1}-x^*}^2 \middle| x^0}& \le \left(1-\sigma_{\min}\left(\frac{\E{\mA_{i:} \mA_{i:}^* }}{\norm{\mA}_F^2}\right)\right)^{t+1} \norm{x^{0}-x^*}^2\\
    & \le \left(1-\frac{\sigma_{\min}\left(\mA^* \mA\right)}{\norm{\mA}_F^2}\right)^{t+1} \norm{x^{0}-x^*}^2.
\end{align*} 
Both the TRK and MRK convergence guarantees depend on the minimal singular value of the expectation over the possible projections onto the rows or row slices for the matrix and tensor versions respectively.

\section{Analysis of TRK in the Fourier domain}\label{sec:Fourier_analysis}

The t-product can be computed efficiently using the Fast Fourier Transform (FFT), since circulant matrices are diagonalized by the DFT. Similarly, we can analyze the convergence of TRK in the Fourier domain and capitalize on the resulting block-diagonal structure. In this section, we present a convergence analysis in the Fourier domain to derive a more interpretable convergence guarantee for TRK. We describe how the TRK update can be performed efficiently in the Fourier domain and additionally demonstrate that TRK is equivalent to performing block MRK on the linear system in the Fourier domain.

\subsection{Notation and preliminary facts}\label{subsec:prelim_Fourier_facts}

We first introduce some additional notation and basic facts that will be used throughout this section. The notation and definitions are adopted from \cite{kilmer2011factorization}. 
Let $\cM\in\C^{m\times \ell \times  n}$ and $\widehat{\cM}$ denote the tensor resulting from applying the DFT matrix to each of the tube fibers of $\cM$. This is operation is referred to in previous literature as a \emph{mode-3 FFT}. Fact 2 of \cite{kilmer2011factorization}, guarantees that
\begin{equation}\label{eqn:fact_2_kilmer}
    \bdiag{\widehat{\cM}}:= \left(\mF_n \otimes \mI_m\right)\bcirc{\cM}
    \left(\mF_n^* \otimes \mI_\ell\right) = 
    \begin{pmatrix}
        \widehat{\mM}_0 &&& \\
        & \widehat{\mM}_1 &&& \\
         && \ddots & \\
        &&& \widehat{\mM}_{n-1}  
    \end{pmatrix},
\end{equation} 
where $\widehat{\mM}_k$ is the $k\thup$ frontal face of $\widehat{\cM}$, $\otimes$ denotes the Kronecker product, $\mF_n$ is the $n\times n$ DFT matrix, and $\bdiag{\widehat{\cM}}$ is the block diagonal matrix formed by the frontal faces of $\widehat{\cM}$. 

We now present several facts which hold true for tensors under the t-product. They will be useful for performing calculations in the Fourier domain.

\begin{fact}\label{fact:Fourier_distributes}
For tensors $\cA$ and $\cB$ the following holds:
    \begin{equation*}
        \bdiag{\widehat{\cA\cB}} = \bdiag{\widehat{\cA}}\bdiag{\widehat{\cB}}.
    \end{equation*}
\end{fact}
\begin{proof}
Let $\cA \in \C^{m\times \ell\times n}$ and $\cB\in \C^{\ell\times p\times n}$, then
\begin{align*}
    \bdiag{\widehat{\cA\cB}}
     &\overset{\eqref{eqn:fact_2_kilmer}}{=} \left(\mF_n\otimes \mI_m\right) \bcirc{{\cA}\cB} \left(\mF_n^*\otimes \mI_p\right)\\
    &\overset{\Cref{fact:bcirc_distr}}{=} \left(\mF_n\otimes \mI_m\right) \bcirc{{\cA}}\bcirc{\cB} \left(\mF_n^*\otimes \mI_p\right)\\
     &= \left(\mF_n\otimes \mI_m\right) \bcirc{{\cA}}\left(\mF_{n}^* \otimes \mI_{\ell}\right)\left(\mF_{n} \otimes \mI_{\ell}\right)\bcirc{\cB} \left(\mF_n^*\otimes \mI_p\right)\\
     &\overset{\eqref{eqn:fact_2_kilmer}}{=} \bdiag{\widehat{\cA}}\bdiag{\widehat{\cB}}.
\end{align*}
\end{proof}

\begin{fact}\label{fact:Fourier_distributes_add}
Addition and~~$\widehat{{\cdot}}$~~are commutative
    \begin{equation*}
        \widehat{\cA + \cB} = \widehat{\cA} + \widehat{\cB}.
    \end{equation*}
\end{fact}
\begin{proof}
Let $\cA \in \C^{m\times \ell\times n}$ and $\cB\in \C^{m \times \ell \times n}$, then
\begin{align*}
    \bdiag{\widehat{\cA + \cB}}
     &\overset{\eqref{eqn:fact_2_kilmer}}{=} \left(\mF_n\otimes \mI_m\right) \bcirc{{\cA} + \cB} \left(\mF_n^*\otimes \mI_{\ell}\right)\\
    &= \left(\mF_n\otimes \mI_m\right)\left( \bcirc{\cA} + \bcirc{\cB}\right) \left(\mF_n^*\otimes \mI_{\ell}\right)\\
     &= \left(\mF_n\otimes \mI_m\right) \bcirc{{\cA}}\left(\mF_{n}^* \otimes \mI_{\ell}\right) + \left(\mF_{n} \otimes \mI_{m}\right)\bcirc{\cB} \left(\mF_n^*\otimes \mI_{\ell}\right)\\
     &\overset{\eqref{eqn:fact_2_kilmer}}{=} \bdiag{\widehat{\cA}} + \bdiag{\widehat{\cB}}.
\end{align*}
\end{proof}

\begin{fact}\label{fact:trans_commutes_w_Fourier}
The conjugate transpose commutes with $\bdiag{~\widehat{\cdot}~}$,
    \begin{equation*}
        \bdiag{\widehat{\cM^*}} = \bdiag{\widehat{\cM}}^*.
    \end{equation*}
    Additionally, if $\bcirc{\cM}$ is symmetric, $\bdiag{\widehat{\cM}}$ is also symmetric.
\end{fact}
\begin{proof}
Let $\cM \in \C^{m \times \ell \times n}$. Then
\begin{align*}
    \bdiag{\widehat{\cM^*}}
    &\overset{\eqref{eqn:fact_2_kilmer}}{=}\left(\mF_n \otimes \mI_{\ell} \right) \bcirc{\cM^*} \left(\mF_n^* \otimes \mI_{m} \right) \\
    &\overset{\Cref{fact:bcirc_transpose}}{=} \left(\mF_n \otimes \mI_{\ell}  \right) \bcirc{\cM}^* \left(\mF_n^* \otimes \mI_{m} \right)\\
    &= \left[\left(\mF_n \otimes \mI_{m} \right) \bcirc{\cM} \left(\mF_n^* \otimes \mI_{\ell} \right)\right]^*\\
    &\overset{\eqref{eqn:fact_2_kilmer}}{=}\bdiag{\widehat{\cM}}^*.
\end{align*}

To see that $\bdiag{\widehat{\cM}}$ is also symmetric when $\bcirc{\cM}$ is symmetric, 
note that
    \begin{align*}
        \bdiag{\widehat{\cM}}^* \overset{\eqref{eqn:fact_2_kilmer}}{=} \left[\left(\mF_n \otimes \mI_m \right) \bcirc{\cM} \left(\mF_n^* \otimes \mI_n \right)\right]^*.
    \end{align*}
\end{proof}

\begin{fact}\label{fact:Fourier_commutes_w_inverse}
The inverse commutes with $\bdiag{~\widehat{\cdot}~}$,
    \begin{equation*}
        \bdiag{\widehat{\cM^{-1}}}=\bdiag{\widehat{\cM}}^{-1}.
    \end{equation*}
\end{fact}
\begin{proof}
    Let $\cM \in \C^{m\times m \times n}$. Note that $\bcirc{\cI_m} = \mI_{mn}$. Using \Cref{fact:bcirc_distr} and \Cref{eqn:fact_2_kilmer},
    \begin{align*}
        \bdiag{\widehat{\cM^{-1}}}&\bdiag{\widehat{\cM}}\\
        &\overset{\eqref{eqn:fact_2_kilmer}}{=}\left(\mF_n \otimes \mI_m \right) \bcirc{\cM^{-1}} \left(\mF_n^* \otimes \mI_m \right)
        \left(\mF_n \otimes \mI_m \right) \bcirc{\cM} \left(\mF_n^* \otimes \mI_m \right) \\
        &=\left(\mF_n \otimes \mI_m \right) \bcirc{\cM^{-1}} \bcirc{\cM} \left(\mF_n^* \otimes \mI_m \right)\\
        &\overset{\Cref{fact:bcirc_distr}}{=}\left(\mF_n \otimes \mI_m \right) \bcirc{\cI_m} \left(\mF_n^* \otimes \mI_m \right)\\
        &= \mI_{mn}.
    \end{align*}
    Analogously, one can show $\bdiag{\widehat{\cM}}\bdiag{\widehat{\cM^{-1}}}= \mI_{mn}$. 
\end{proof}

\subsection{A more interpretable convergence guarantee}\label{subsec:Fourier_equiv}

Using Fact 2 of \cite{kilmer2011factorization}, we can derive a more interpretable convergence guarantee in terms of the tensor $\cA$. Specifically, assuming that the indices $i_t$ are sampled uniformly at random at each iteration, we can restate \Cref{thm:conv} as follows.

\begin{theorem}\label{thm:Fourier_conv}
Let $\cX^*$ be such that $\cA \cX^* = \cB$ and $\cX^t$ be the $t\thup$ approximation of $\cX^*$ given by the updates of \Cref{eqn:trk_update} with initial iterate $\cX^0$ and indices $i_t\in[m-1]$ sampled uniformly at random at each iteration. The expected error at the $(t+1)^{st}$ iteration satisfies 
\begin{equation*}
    \E{\norm{\cX^{t+1}-\cX^*}_F^2\middle| \cX^{0}} \le \left(1-\min_{k\in [n-1]} \frac{\sigma^2_{\min}\left( \widehat{\cA}_k \right)}{m\norm{ \widehat{\cA}_k}_{\infty,2}^2}\right)^{t+1} \norm{\cX^{0}-\cX^*}_F^2,
\end{equation*}
where $\|\cdot \|_{\infty, 2}$ is as defined in \Cref{eqn:max_Fourier_norm}, $\widehat{\cA}_k$ is the $k^{th}$ frontal slice of $\widehat{\cA}$, and $\sigma_{min}(\cdot)$ denotes the smallest singular value.
\end{theorem}

\begin{proof}
    Let $\widehat{\cP_i}$ be the tensor formed by applying FFTs to each tube fiber of $\cP_i = \cA_{i::}^* \left(\cA_{i::}\cA_{i::}^*\right)^{-1}\cA_{i::}$. By \Cref{eqn:fact_2_kilmer}, we have that \begin{equation*}
        \bdiag{\widehat{\cP_i}} = 
        \left(\mF_{n} \otimes \mI_{\ell}\right){\bcirc{\cP_i}}\left(\mF_{n}^* \otimes \mI_{\ell}\right),
    \end{equation*} is a block diagonal matrix with blocks $\left(\widehat{\mP_i}\right)_k$, where $\left(\widehat{\mP_i}\right)_k$ is the $k\thup$ frontal slice of the tensor $\widehat{\cP_i}$. 
    We note that the projected error in \Cref{eqn:innerprod} can be rewritten as
    \begin{align*}
        &\E{\norm{\cP_i\cE^t}_F^2}
        = \sum _{j=1}^{p}\langle \E{\bcirc{\cP_i}}\unfold{\cE^t}_{:j},\unfold{\cE^t}_{:j}\rangle\\
        &= \sum _{j=1}^{p}\E{\langle \left(\mF_{n} \otimes \mI_{\ell}\right){\bcirc{\cP_i}}\left(\mF_{n}^* \otimes \mI_{\ell}\right)
        \left(\mF_{n} \otimes \mI_{\ell}\right)\unfold{\cE^t}_{:j},  \left(\mF_{n} \otimes \mI_{\ell}\right)\unfold{\cE^t}_{:j}\rangle}\\
        &= \sum _{j=1}^{p}\E{\langle \bdiag{\widehat{\cP_i}}
        \left(\mF_{n} \otimes \mI_{\ell}\right)\unfold{\cE^t}_{:j},  \left(\mF_{n} \otimes \mI_{\ell}\right)\unfold{\cE^t}_{:j}\rangle}.
    \end{align*}
    
    Now, since $\E{\bdiag{\widehat{\cP_i}}}$ is symmetric by \Cref{fact:trans_commutes_w_Fourier}, 
    
    \begin{equation}\label{eqn:unsimp_conv_bound}
        \E{\norm{\cP_i\cE^t}_F^2}
         \ge \sigma_{\min}\left(\E{\bdiag{\widehat{\cP_i}}}\right)  \norm{\left(\mF_{n} \otimes \mI_{\ell}\right)\unfold{\cE^t}}_F^2.
    \end{equation}
    
    Note that,
    \begin{align*}
        \norm{\left(\mF_{n} \otimes \mI_{\ell}\right)\unfold{\cE^t}}_F^2
        & = \sum _{j=1}^{p}
        \langle \left(\mF_{n} \otimes \mI_{\ell}\right)\unfold{\cE^t}_{:j},\left(\mF_{n} \otimes \mI_{\ell}\right)\unfold{\cE^t}_{:j}\rangle\\
        & = \sum _{j=1}^{p}
        \langle \unfold{\cE^t}_{:j},\unfold{\cE^t}_{:j}\rangle\\
        & = \norm{\unfold{\cE^t}}_F^2\\
        & = \norm{\cE^t}_F^2.
    \end{align*}

Since $\bdiag{\widehat{\cP_i}}$ is block diagonal, 
\begin{equation*}
    \sigma_{\min}\left(\E{\bdiag{\widehat{\cP_i}}}\right) 
    =
    \min_{k \in [n-1]}  \sigma_{\min}\left(\E{\left(\widehat{\mP_i}\right)_k}\right). 
\end{equation*}

Factoring $\bdiag{\widehat{\cP_i}}$, 
\begin{align*}
     \bdiag{\widehat{\cP_i}} \overset{\Cref{fact:Fourier_distributes}}{=} \bdiag{\widehat{{\cA}_{i::}^*}}\bdiag{\reallywidehat{\left({\cA}_{i::}{\cA}_{i::}^*\right)^{-1}}}\bdiag{\widehat{{\cA}_{i::}}}\\ 
     \overset{\Crefrange{fact:trans_commutes_w_Fourier}{fact:Fourier_commutes_w_inverse}}{=} \bdiag{\widehat{{\cA}_{i::}}}^*\bdiag{\reallywidehat{{\cA}_{i::}{\cA}_{i::}^*}}^{-1}\bdiag{\widehat{{\cA}_{i::}}}\\ 
     \overset{\Cref{fact:Fourier_distributes}}{=} \bdiag{\widehat{{\cA}_{i::}}}^*\left[\bdiag{\widehat{{\cA}_{i::}}}\bdiag{\widehat{{\cA}_{i::}^*}}\right]^{-1}\bdiag{\widehat{{\cA}_{i::}}}.
\end{align*}

Noting that $\bdiag{\widehat{{\cA}_{i::}}}\bdiag{\widehat{{\cA}_{i::}^*}}$ is a diagonal matrix, one can see that $\left(\widehat{\mP_i}\right)_k$ is the projection onto $\left(\widehat{{\cA}_{i::}}\right)_k$ by rewriting the $k\thup$ frontal face of $\widehat{\cP_i}$ as 
\begin{equation*}
    \left(\widehat{\mP_i}\right)_k = \frac{\left(\widehat{{\cA}_{i::}}\right)^*_k\left(\widehat{ {\cA}_{i::}}\right)_k}{\left(\widehat{{\cA}_{i::}}\widehat{{\cA}_{i::}^*}\right)
    _k}.
\end{equation*}

We can thus rewrite \Cref{eqn:unsimp_conv_bound} as 
\begin{equation}\label{eqn:conv_bound_w_min}
    \E{\norm{\cP_i\cE^t}_F^2}
     \ge \min_{k\in[n-1]} \sigma_{\min}\left(\E{\frac{\left(\widehat{{\cA}_{i::}}\right)^*_k\left(\widehat{ {\cA}_{i::}}\right)_k}{\left(\widehat{{\cA}_{i::}}\widehat{{\cA}_{i::}^*}\right)_k}}\right) \norm{\cE^t}_F^2.
\end{equation}

The expectation of \Cref{eqn:unsimp_conv_bound} can now be calculated explicitly. 
For simplicity, we assume that the row indices $i$ are sampled uniformly. As in MRK extensions and literature, many other sampling distributions could be used. 

To derive a lower bound for the smallest singular value in \Cref{eqn:conv_bound_w_min}, define
\begin{equation}\label{eqn:max_Fourier_norm}
    \norm{\widehat{\cA}_k}_{\infty,2}^2 := \max_{i}\left[\left(\widehat{{\cA}_{i::}}\widehat{{\cA}_{i::}^*}\right)_k\right].
\end{equation} 
The values $\left(\widehat{{\cA}_{i::}}\widehat{{\cA}_{i::}^*}\right)_k$ are necessarily positive for all $k \in [n-1]$  under \Cref{assump:invertibility_req} as 
\begin{align*}
    \left(\widehat{{\cA}_{i::}}\widehat{{\cA}_{i::}^*}\right)_k
    &=
    \bdiag{\widehat{\cA_{i::}}\widehat{\cA_{i::}^*}}_{kk} \\
    &= \bdiag{\widehat{\cA_{i::}}}_k \bdiag{\widehat{\cA_{i::}^*}}_k\\
    &= \left(\mF_n\right)_{k:} \bcirc{\cA_{i::}}\bcirc{\cA_{i::}^*} \left(\mF_n\right)_{k:}^* \\
    &= \left(\mF_n\right)_{k:} \bcirc{\cA_{i::}}\bcirc{\cA_{i::}}^* \left(\mF_n\right)_{k:}^*\\
    &= \norm{\bcirc{\cA_{i::}}^* \left(\mF_n\right)_{k:}^*}_2^2.
\end{align*}

Now, it can be easily verified that
\begin{align*}
    \sigma_{\min}\left(\E{\frac{\left(\widehat{{\cA}_{i::}}\right)_k^*\left(\widehat{ {\cA}_{i::}}\right)_k}{\left(\widehat{{\cA}_{i::}}\widehat{{\cA}_{i::}^*}\right)_k}}\right)
    & \ge \sigma_{\min}\left( \frac{1}{m} \sum_{i=0}^{m-1} \frac{\left(\widehat{{\cA}_{i::}}\right)_k^*\left(\widehat{ {\cA}_{i::}}\right)_k}{\norm{\widehat{\cA}_k}_{\infty,2}^2}\right)\\
    & = \frac{\sigma^2_{\min}\left( \widehat{\cA}_k \right)}{m\norm{\widehat{\cA}_k}_{\infty,2}^2} .
\end{align*}

The projected error of \Cref{eqn:conv_bound_w_min} then becomes
\begin{equation}\label{eqn:closed_form_exp}
    \E{\norm{\cP_i\cE^t}_F^2}
     \ge \min_{k\in [n-1]} \frac{\sigma^2_{\min}\left( \widehat{\cA}_k \right)}{m\norm{ \widehat{\cA}_k}_{\infty,2}^2} \norm{\cE^t}_F^2,
\end{equation}
leading to a contraction coefficient of
\begin{equation}\label{eqn:rho_Fourier}
    \rho = 1 - \min_{k\in [n-1]} \frac{\sigma^2_{\min}\left( \widehat{\cA}_k \right)}{m\norm{ \widehat{\cA}_k}_{\infty,2}^2}.
\end{equation}
We can thus rewrite the guarantee in \Cref{thm:conv} for uniform random sampling of the row indices $i$ as 
\begin{equation}\label{eqn:Fourier_conv}
    \E{\norm{\cX^{t+1}-\cX^*}_F^2\middle| \cX^{0}} \le \left(1-\min_{k\in[n-1]} \frac{\sigma^2_{\min}\left( \widehat{\cA}_k \right)}{m\norm{ \widehat{\cA}_k}_{\infty,2}^2}\right)^{t+1} \norm{\cX^{0}-\cX^*}_F^2.
\end{equation}
\end{proof}

\subsection{Equivalence of TRK and block MRK applied in the Fourier domain}\label{subsec:block_equiv}
In this section, we observe a connection between the proposed TRK method and the previously studied block MRK algorithm~\cite{needell2014paved}. This analysis helps to further bridge the understanding of connections between TRK and MRK. 
In block MRK, one projects the current iterate onto the solution space of a set of constraints (set of rows of the linear system) as opposed to the solution space with respect to a single row. In practice, block MRK can lead to a significant speed up over MRK \cite{needell2014paved}. 

Here we show the equivalence of TRK and block MRK performed in the Fourier domain with specific block partitions and remark on the convergence rate implications in the block MRK setting. Using \Cref{eqn:fact_2_kilmer}, the tensor linear system \Cref{eqn:linsys} can be rewritten as: 
\begin{equation}
    \begin{pmatrix}
        \widehat{\cA}_0 &&& \\
        & \widehat{\cA}_1 &&& \\
         && \ddots & \\
        &&& \widehat{\cA}_{n-1}
    \end{pmatrix}
    \begin{pmatrix}
        \widehat{\cX}_0 \\
        \widehat{\cX}_1  \\
         \vdots  \\
        \widehat{\cX}_{n-1} 
    \end{pmatrix}
    =
        \begin{pmatrix}
        \widehat{\cB}_0 \\
        \widehat{\cB}_1  \\
         \vdots  \\
        \widehat{\cB}_{n-1} 
    \end{pmatrix}.
    \label{eqn:Fouriersys}
\end{equation}

The system shown in \Cref{eqn:Fouriersys} can be solved using block MRK such that the resulting iterate is equivalent to the TRK iterate in the following way. Let 
\begin{equation}
    \label{eqn:tau}
    \tau_{i} = \{ km + i~|~ k\in [n-1] \},
\end{equation} 
denote in set of indices corresponding to a randomly selected block of the measurement matrix in~\Cref{eqn:Fouriersys}. This choice of $\tau_{i}$ corresponds to selecting the $i\thup$ row of each $\widehat{\cA}_k$ in $\bdiag{\widehat{\cA}}$, i.e., each row of $\widehat{\cA}_{i::}$ appears along the diagonal of $\bdiag{\widehat{\cA}}_{\tau_{i}}$ 
and therefore, $\bdiag{\widehat{\cA}}_{\tau_{i}} = \bdiag{\widehat{\cA}_{i::}}$.

For a randomly selected row index $i_t \in [m-1]$, the block MRK update for \Cref{eqn:Fouriersys} is aptly written as:
    \begin{align}
    \unfold{\widehat{\cX}^{t+1}} &=  \unfold{\widehat{\cX}^t} - \bdiag{\widehat{\cA}}_{\tau_{i_t}}^\dagger\left(\bdiag{\widehat{\cA}}_{\tau_{i_t}}  \unfold{\widehat{\cX}^t}- \unfold{\widehat{\cB}}_{\tau_{i_t}}\right)  \label{eqn:trk_brk_iterate} \\
    &=  \unfold{\widehat{\cX}^t} - \bdiag{\widehat{\cA}_{i_t::}}^\dagger\left(\bdiag{\widehat{\cA}_{i_t::}} \unfold{\widehat{\cX}^t}- \unfold{\widehat{\cB}}_{\tau_{i_t}}\right).  \nonumber 
    \end{align}
Using \Cref{eqn:fact_2_kilmer} and \Cref{fact:Fourier_distributes,fact:Fourier_distributes_add,fact:Fourier_commutes_w_inverse,fact:trans_commutes_w_Fourier}, we can show
\begin{align*}
    \bdiag{\widehat{\cA}_{i_t::}}^\dagger &= \bdiag{\widehat{\cA}_{i_t::}}^*\left(\bdiag{\widehat{\cA}_{i_t::}} \bdiag{\widehat{\cA}_{i_t::}}^* \right)^{-1} \\ 
    &= \bdiag{{\widehat{\cA}_{i_t::}^*} \left({\widehat{\cA}_{i_t::}}{\widehat{\cA}_{i_t::}^*}\right)^{-1}}
\end{align*}
Therefore, noting the following equalities and folding the right and left sides of the equation into tensors, we derive the iterate update for $\widehat{\cX}^{t+1}$ from the block MRK update: 
\begin{align}
    \unfold{\widehat{\cX}^{t+1}} &=  \unfold{\widehat{\cX}^t} - \bdiag{{\widehat{\cA^*}_{i::}} \left({\widehat{\cA}_{i_t::}}{\widehat{\cA^*}_{i_t::}} \right)^{-1}}\left(\bdiag{\widehat{\cA}_{i_t::}} \unfold{\widehat{\cX}^t}- \unfold{\widehat{\cB}}_{\tau_{i_t}}\right)\nonumber  \\ 
    &=  \unfold{\widehat{\cX}^t} - (\mF_n \otimes \mI_\ell) \unfold{{{\cA}_{i_t::}}^* \left({{\cA}_{i_t::}}{{\cA}_{i_t::}}^* \right)^{-1}\left( {\cA}_{i_t::} {\cX}^t - {\cB}_{i_t::} \right)}  \nonumber \\
    &=  \unfold{\widehat{\cX}^t} - \unfold{{\widehat{\cA^*}_{i_t::}} \left({\widehat{\cA}_{i_t::}}{\widehat{\cA^*}_{i_t::}} \right)^{-1}\left( \widehat{\cA}_{i_t::} \widehat{\cX}^t - \widehat{\cB}_{i_t::} \right)}  \nonumber
\end{align}
    \begin{equation}
    \Rightarrow \widehat{\cX}^{t+1} =  \widehat{\cX}^t - \widehat{\cA^*}_{i_t::} \left(\widehat{\cA}_{i_t::}\widehat{\cA^*}_{i_t::}\right)^{-1}\left(\widehat{\cA}_{\tau_{i_t::}} \widehat{\cX}^t-\widehat{\cB}_{i_t::}\right). 
    \label{eqn:trk_update_fft}
\end{equation}

Since the FFT is applied to each tube fiber of $\cA$ independently, $\widehat{\cA_{i::}} = \widehat{\cA}_{i::}$. 
To see that \Cref{eqn:trk_update_fft} is equivalent to \Cref{eqn:trk_update} one can use \Cref{fact:Fourier_distributes,fact:Fourier_distributes_add,fact:Fourier_commutes_w_inverse,fact:trans_commutes_w_Fourier} to show that $\widehat{\cX}^{t+1} = \widehat{\cX^{t+1}}$, that is taking the inverse FFT on the tubes of $\widehat{\cX^{t+1}}$ will return the TRK update \Cref{eqn:trk_update}.

\begin{remark} The contraction rate for block MRK applied to the linear system ~\Cref{eqn:Fouriersys} with iterates as shown in \Cref{eqn:trk_brk_iterate} is 
\begin{equation}
 \rho_{_{\text{BRK}}} = 1 - \frac{ \sigma^2_{\min}\left(\bdiag{\widehat{\cA}}\right)}{mn \max_{i} \lambda_{\max} \left(\bdiag{\widehat{\cA}}_{\tau_{i}}\bdiag{\widehat{\cA}}^*_{\tau_{i}}\right)}.
 \label{eqn:rho_brk}
\end{equation}
The contraction coefficient $\rho_{_{\text{BRK}}}$ is a direct result of the theoretical guarantees for block MRK shown in \cite{needell2014paved}.
Note that due to the block-diagonal structure, the numerator of the second term of \Cref{eqn:rho_brk} can be simplified to
\begin{align*}
    \sigma^2_{\min}\left(\bdiag{\widehat{\cA}}\right) &=  \min_{ k \in [n-1]} \sigma^2_{\min} (\widehat{\cA}_k).
\end{align*}
Using the fact that $\bdiag{\widehat{\cA}}_{\tau_{i}} = \bdiag{\widehat{\cA}_{{i}::}}$ along with \Cref{fact:Fourier_distributes,fact:trans_commutes_w_Fourier}, it can be easily shown that $\bdiag{\widehat{\cA}}_{\tau_{i}}\bdiag{\widehat{\cA}}^*_{\tau_{i}} = \bdiag{\widehat{\cA}_{{i}::}\widehat{\cA^*}_{{i}::}}$. Thus, the denominator of \Cref{eqn:rho_brk} can be simplified to: 
\begin{align*}
    \max_{i} \lambda_{\max} \left(\bdiag{\widehat{\cA}}_{\tau_{i}}\bdiag{\widehat{\cA}}^*_{\tau_{i}}\right) &=  \max_{i} \lambda_{\max} \left(\bdiag{\widehat{\cA}_{{i}::}\widehat{\cA}_{{i}::}^*} \right) \\
    &= \max_{i} \max_k \left[ \widehat{\cA}_{{i}::}\widehat{\cA}_{{i}::}^* \right]_k \\
    \overset{\eqref{eqn:max_Fourier_norm}}{=} \max_k \norm{\widehat{\cA}_k}_{\infty,2}^2,
\end{align*}
where the norm in the last equality is as defined in~\Cref{eqn:max_Fourier_norm}.
Putting this all together, the contraction rate for block MRK applied to~\Cref{eqn:Fouriersys} is 
\begin{equation}
 \rho_{_{\text{BRK}}} = 1 - \frac{\min_k \sigma^2_{\min} (\widehat{\cA}_k)}{mn \max_{k} \norm{\widehat{\cA}_k}_{\infty,2}^2}.
    \label{eqn:blockRKRate}
    \end{equation}

Compared to the convergence rate derived for TRK in~\Cref{thm:Fourier_conv}, the standard block MRK convergence guarantee is weaker (slower). The standard analysis for the convergence of block MRK is not restricted to block diagonal systems. Thus, although block MRK applied to \Cref{eqn:Fouriersys} with predetermined blocks $\tau_{i}$ is equivalent to the proposed TRK update, the standard block MRK guarantee is weaker since the TRK analysis takes advantage of the block diagonal structure of the system in the Fourier domain. 
\end{remark}

\begin{remark}
The block-diagonal system in \Cref{eqn:Fouriersys} is highly parallelizable. Specifically, each component block of the system $\widehat{\cA}_k\widehat{\cX}_k=\widehat{\cB}_k$ for ${k\in[n-1]}$ can be solved independently. For $m$ extremely large, however, loading a single $\widehat{\cA}_i$ into memory maybe be impossible. In such settings, a randomized iterative method such as TRK is advantageous. The block-diagonal structure of the subsampled system in the Fourier domain also allows the update for each component block to be computed in parallel.
\end{remark}

Making use of the equivalence of TRK and block MRK in the Fourier domain, TRK can be implemented efficiently using methods for matrices as detailed in \Cref{algo:fourier_TRK}. Note that \Cref{eqn:trk_brk_iterate} can be reformulated as 
\begin{equation*}
    \widehat{\cX^{t+1}}_k =  \widehat{\cX^t}_k - \left(\widehat{\cA}_{i_t:k}\right)^\dagger\left(\widehat{\cA}_{i_t:k}  \widehat{\cX}^t_k- \widehat{\cB}_{i_t:k}\right) \mbox{ for } k\in [n-1],
\end{equation*}
by making use of the block structure.

\begin{algorithm}[t]
\begin{algorithmic}
\State \textbf{Input:}  $\cX^0\in \C^{\ell\times p\times n},$ $\cA\in \C^{m\times \ell \times n}$, $\cB\in\C^{m\times p \times n}$,  and probabilities $p_0, \dots, p_{m-1}$ corresponding to each row slice of $\cA$
\State Compute $\widehat{\cX^0}, \widehat{\cA}, \widehat{\cB}$ as in \Cref{eqn:fact_2_kilmer}
\For {$t = 0, 1, 2, \dots$}
    \State  Sample $i_t \sim p_i$
    \For {$k = 0, 1, \dots, n-1$}
    	\State  $ \widehat{\cX^{t+1}}_k =  \widehat{\cX^t}_k - \left(\widehat{\cA}_{i_t:k}\right)^\dagger\left(\widehat{\cA}_{i_t:k}  \widehat{\cX}^t_k- \widehat{\cB}_{i_t:k}\right)$
	\EndFor
\EndFor
\State Recover $\cX^{t+1}$ from $\widehat{\cX^{t+1}}$
\State \textbf{Output:} last iterate $\cX^{t+1}$
\end{algorithmic}

\caption{Tensor RK computed in the Fourier domain}
\label{algo:fourier_TRK}
\end{algorithm}

\begin{remark} 
The equivalence between TRK and block MRK with blocks indexed by \Cref{eqn:tau} also reveal a straightforward analysis for the comparison of the computational complexity between TRK and MRK. The per iteration complexity of MRK using rows $ \textbf{A}_{i:} \in \mathbb{R}^{1 \times \ell n} $ is $\mathcal{O}(\ell n)$ and the per iteration complexity of TRK using rows $\cA_{i::} \in \mathbb{R}^{1 \times \ell \times n}$ is $\mathcal{O}(\ell n^2)$. 
\end{remark}

\section{Experiments}

In this section, we present numerical experiments comparing MRK and TRK. The implementation of the TRK algorithm used is as outlined in~\Cref{algo:TRK}, unless otherwise noted. First, we show empirically that with an increasing number of measurements $m$, the contraction coefficient for TRK is smaller than that of MRK indicating a stronger convergence guarantee. Next, we compare the performance of TRK with that of MRK applied to a matrix linear system where the memory complexity of the measurement matrix is preserved. Then, we move on to the setting in which one is given tensor measurements $\cB$ and compare the performance of TRK with that of MRK applied to the unfolded tensor system 
\begin{equation*}
    \bcirc{\cA} \unfold{\cX} = \unfold{\cB}.
\end{equation*}
These experiments demonstrate the computational benefits of using TRK given by \Cref{eqn:trk_update} over applying standard MRK to an unfolded system.

\subsection{Contraction coefficients of TRK and MRK}
In this experiment, the contraction coefficient of the proposed TRK is compared to that of MRK. In order to apply the standard MRK method to recover the three-dimensional signal $\cX$, we unfold the tensor $\cX$ into the matrix $\unfold{\cX}\in \C^{\ell n \times p}$ and collect measurements $\mB\in\C^{\mu\times p}$ of the signal $\cX$ via the measurement matrix $\mA\in \C^{\mu \times n\ell}$, resulting in the matrix linear system
\begin{equation}\label{eqn:mat_system}
    \mA \unfold{\cX} = \mB.
\end{equation}
After each iteration of MRK applied to \Cref{eqn:mat_system}, the iterate $\unfold{\cX^{t+1}}$ satisfies
\begin{equation}\label{eqn:mat_iter_constr}
    \mA_{i_t:} \unfold{\cX^{t+1}} = \unfold{\cB}_{i_t:}.
\end{equation}
Thus, the constraint is applied to each column of $\unfold{\cX}$ or equivalently each column slice of $\cX$ independently. 
Note that the measurement matrix $\mA$ will have the same number of elements as the measurement tensor $\cA$ in \Cref{eqn:linsys} if $\mu = m$. 

Assuming that the rows of $\mA$ are normalized, MRK applied to matrix linear systems has a contraction coefficient of 
\begin{equation}\label{eqn:mat_contr}
    1 - \sigma_{min}^2(\mA)/m.
\end{equation}
For TRK, the contraction coefficient from~\Cref{thm:Fourier_conv} is 
\begin{equation*}
 1 - \min_{k \in [n-1]} \frac{\sigma^2_{\min}\left( \widehat{\cA}_k \right)}{m\norm{ \widehat{\cA}_k}_{\infty,2}^2}.
\end{equation*}
In this experiment, row slices $\cA_{i_::}$ have unit Frobenius norm and indices $i\in [m-1]$ are selected uniformly at random at each iteration. 

The measurement matrix $\mA\in \C^{m \times \ell n}$ and measurement tensor $\cA\in \C^{m \times \ell \times n}$ are generated as follows. The entries of $\mA \in \C^{m \times \ell n}$ are drawn i.i.d.\ from a standard Gaussian distribution and then each row is normalized to have unit norm. The entries of $\cA \in \mathbb{R}^{m \times \ell \times n}$ are also drawn i.i.d.\ from a standard Gaussian distribution but row slices $\cA_{i::}$ (as opposed to matrix rows) of $\cA$ are normalized to have unit Frobenius norm. Note that both the tensor $\cA$ and matrix $\mA$ in this experiment have the same memory complexity of $\mathcal{O}(m \ell n)$. The contraction coefficients, computed via \Cref{eqn:mat_contr} for matrices $\mA$ and \Cref{eqn:rho_Fourier} for tensors $\cA$, with a varying number of measurements $m$ are presented in Figure~\ref{fig:contraction}. Here, the dimensions $\ell = 20$ and $n = 10$ are fixed. For each number of measurements $m$, the contraction coefficients are averaged over 50 random realizations of the measurement tensor or matrix.

In this experiment, the contraction coefficients for MRK and TRK differ, with TRK being smaller (i.e., faster convergence) for larger $m$. Thus, in the large-scale setting where $m \gg \ell n$, TRK is expected to converge faster than MRK, as we will see in the experiments of \Cref{subsec:empirical_perf}. 
When a small number of measurements $m$ are used, MRK has a smaller contraction coefficient than TRK, however, we are primarily concerned with the setting in which $m \gg \ell n$ as this is the typical use case for Kaczmarz methods.

\begin{figure}[h!]
    \centering
    \includegraphics[width=.45\textwidth]{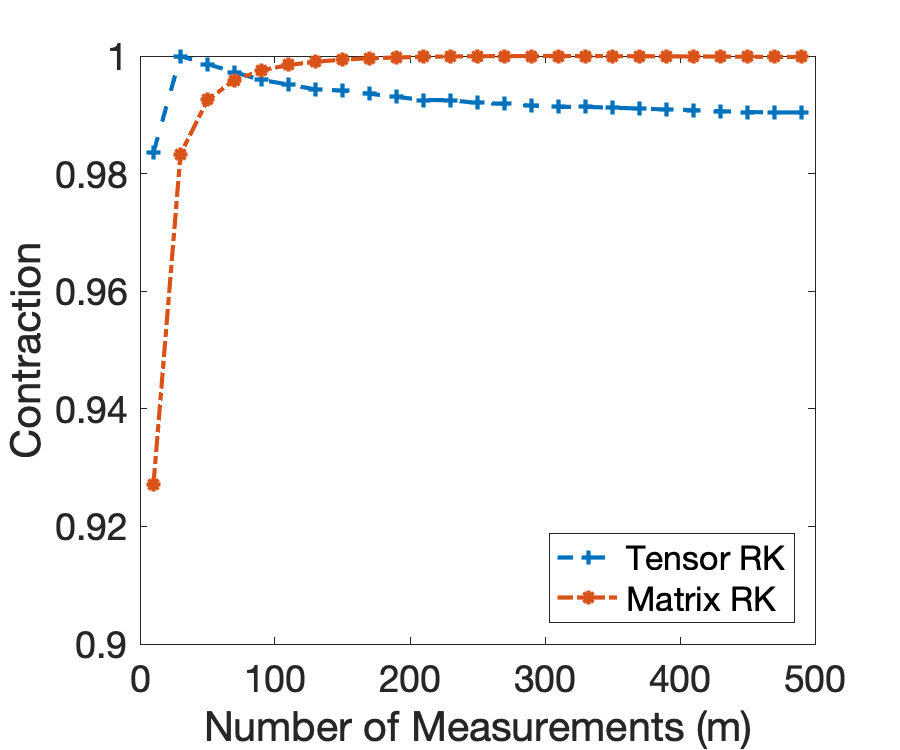}
    \caption{Comparison between contraction coefficients of MRK (\Cref{eqn:mat_contr}) applied to a matrix linear system and TRK (\Cref{eqn:rho_Fourier}) applied to a tensor system.}
    \label{fig:contraction}
\end{figure}

\subsection{Empirical performance of TRK and MRK}\label{subsec:empirical_perf} 
We now compare the empirical performance of MRK and TRK on linear systems  $\cA\cX = \cB$ and $\mA  \mX = \mY $. Similar to the previous experiment, the dimensions of $\cA$ and $\mA$ are selected to require a similar measurement complexity while solving for unknown signals of comparable dimensions. More specifically, for the tensor system we have $\cA \in \C^{m \times \ell \times n}$ and $\cX \in \C^{\ell \times p \times n}$, while for the matrix system, we have $\mA \in \C^{m \times \ell n}$ and $\mX \in \C^{\ell n \times p}$. The entries of $\cA$ and $\mA$ are initialized with i.i.d.\ standard Gaussian entries then normalized to have unit row slice and matrix row Frobenius norm respectively. The entries of the signals $\cX$ and $\mX$ are drawn i.i.d.\ from a standard Gaussian distribution and the empirical results presented here are averaged over 20 random runs of TRK and MRK. For TRK, we use the implementation outlined in Algorithm~\ref{algo:fourier_TRK}.

\Cref{fig:exp_fixedmemory_500} compares the empirical performances of the two algorithms for an over-determined system with $m = 500$, $\ell = 20$, $n = 10$, and $p = 10$. We refer to a tensor linear system as over-determined if the Fourier transformed systems of~\Cref{eqn:Fouriersys} is over-determined, i.e., if $m \geq \ell$. In the over-determined setting, we plot the convergence of the algorithms with respect to iterations (left plot) as well as CPU time (right plot). We observe that in both settings, TRK outperforms MRK in terms of iterations and CPU times. While, visually, MRK does not seem to be making progress towards the solution in either setting, it is in fact converging slowly. This should not be surprising given the equivalence between TRK and block MRK. In particular, one can think of TRK as block MRK acting on $n$ rows at a time (whereas MRK only works on on row at a time). 

\begin{figure}[h!]
    \centering
    \includegraphics[width=.45\textwidth]{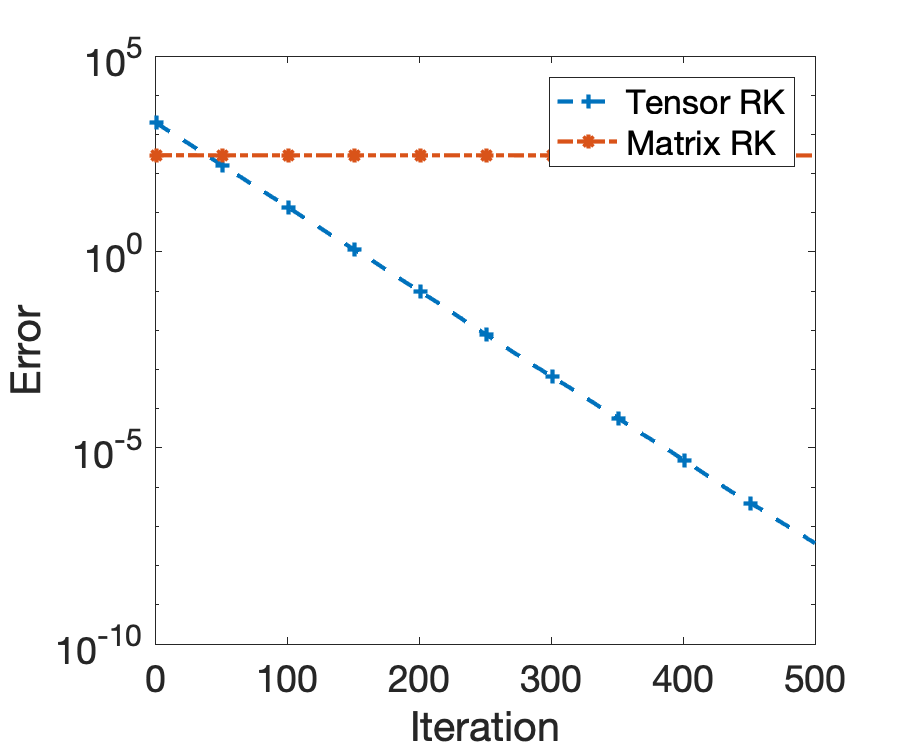}
    \includegraphics[width=.45\textwidth]{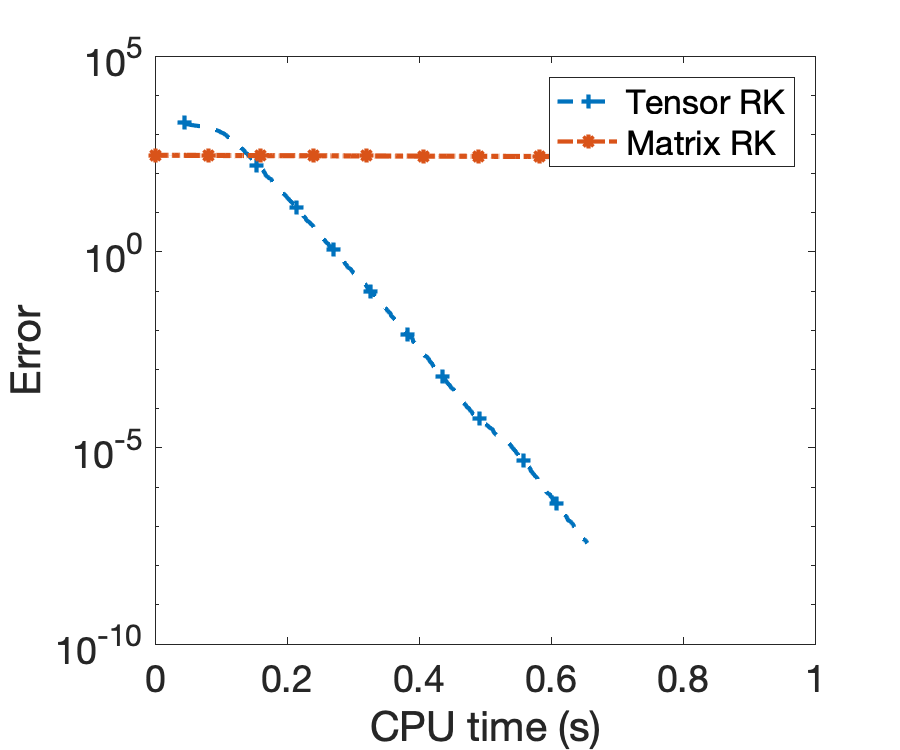}
    \caption{Comparison MRK and TRK when the measurement matrix (or tensor) has a fixed memory budget of $\mathcal{O}(m \ell n)$ bits when $m = 500$, $\ell = 20$, and $n = 10$.} 
    \label{fig:exp_fixedmemory_500}
\end{figure}

We additionally consider the setting in which one is immediately provided the measurement tensor $\cA\in \C^{m \times \ell \times n}$ and corresponding measurements $\cB \in \C^{m \times p \times n}$ and can choose between performing signal recovery using TRK or by unfolding the tensor system and solving $\bcirc{\cA} \unfold{\cX} = \unfold{\cB}$ using MRK.

The tensor $\cA$ is initialized with i.i.d.\ standard Gaussian entries and the measurement matrix $\mA$ is taken to be $\mA = \bcirc{\cA}$. Here, $m = 100$, $\ell = 15$, $ n = 10$, and $p = 30$. Figure~\ref{fig:exp1_nonoise} plots the resulting empirical performance averaged over 20 random runs of TRK and MRK when choosing between signal recovery using TRK or MRK for a given tensor measurement system. We again see that TRK converges at a much faster rate than MRK in this setting.

\begin{figure}
    \centering
    \includegraphics[width=.45\textwidth]{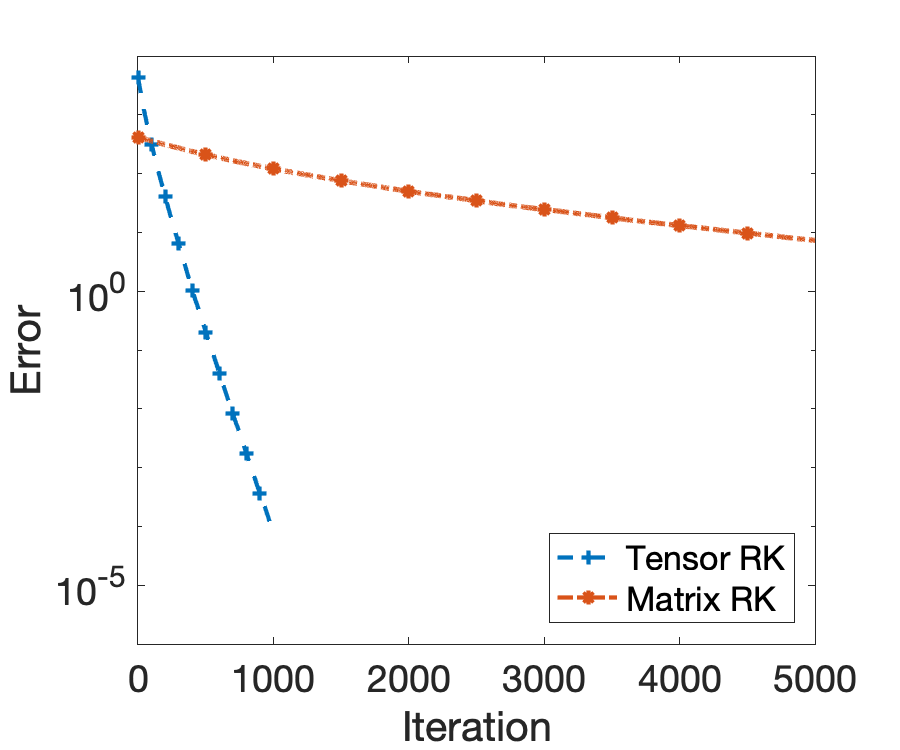}
    \includegraphics[width=.45\textwidth]{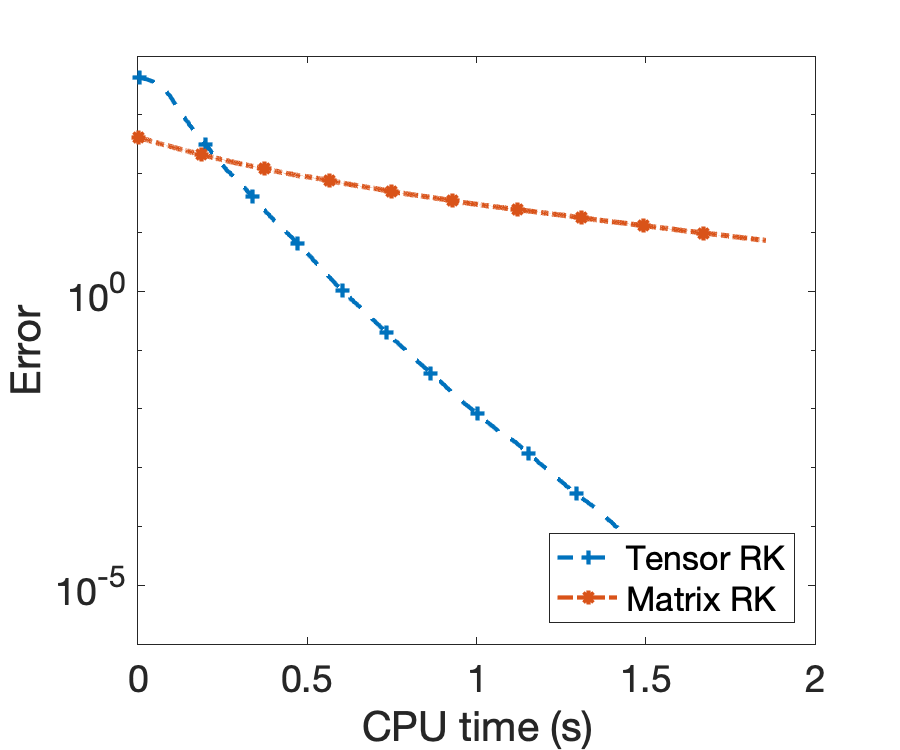}
    \caption{Performance comparison of MRK on matricized linear system and TRK on tensor linear system.}
    \label{fig:exp1_nonoise}
\end{figure}

\subsection{Empirical performance of TRK and block MRK}
To support the theoretical guarantees and remarks regarding the equivalence of TRK and block MRK, experimental results comparing the empirical performance of the two algorithms are presented in this section. In \Cref{fig:exp_trk_brk}, TRK and block MRK are used to solve a tensor linear system as shown in~\Cref{eqn:linsys}. TRK solves the tensor system via the update  in~\Cref{eqn:trk_update} while block MRK is performed on the transformed system in the Fourier domain given in~\Cref{eqn:Fouriersys} with predetermined blocks $\tau_{i} = \{ km + i~|~k\in[n-1] \}$. The measurement tensor $\cA \in \mathbb{R}^{100 \times 30 \times 5}$  and signal tensor $\cX \in \mathbb{R}^{30 \times 15 \times 5}$ contains i.i.d.\ standard Gaussian entries. All approximation errors are averaged over 20 runs of the respective algorithm. The theoretical upper bounds, titled in the legend with `UB', are computed using \Cref{eqn:Fourier_conv} for TRK and \Cref{eqn:blockRKRate} for block MRK. \Cref{fig:exp_trk_brk} clearly shows that TRK and block MRK perform similarly across iterations as expected since the two methods are shown to be equivalent in \Cref{subsec:block_equiv}. As remarked, the TRK upper bound shown in Theorem~\ref{thm:Fourier_conv} has a slight advantage over the general block MRK convergence guarantees as these do not make use of the block diagonal structure of~\Cref{eqn:Fouriersys}. Experiments comparing CPU times for TRK and block MRK are omitted, as the two methods are equivalent as shown in~\Cref{subsec:block_equiv} and highly optimized algorithms exist for the matrix implementation.

\begin{figure}[h!]
    \centering
    \includegraphics[width=.45\textwidth]{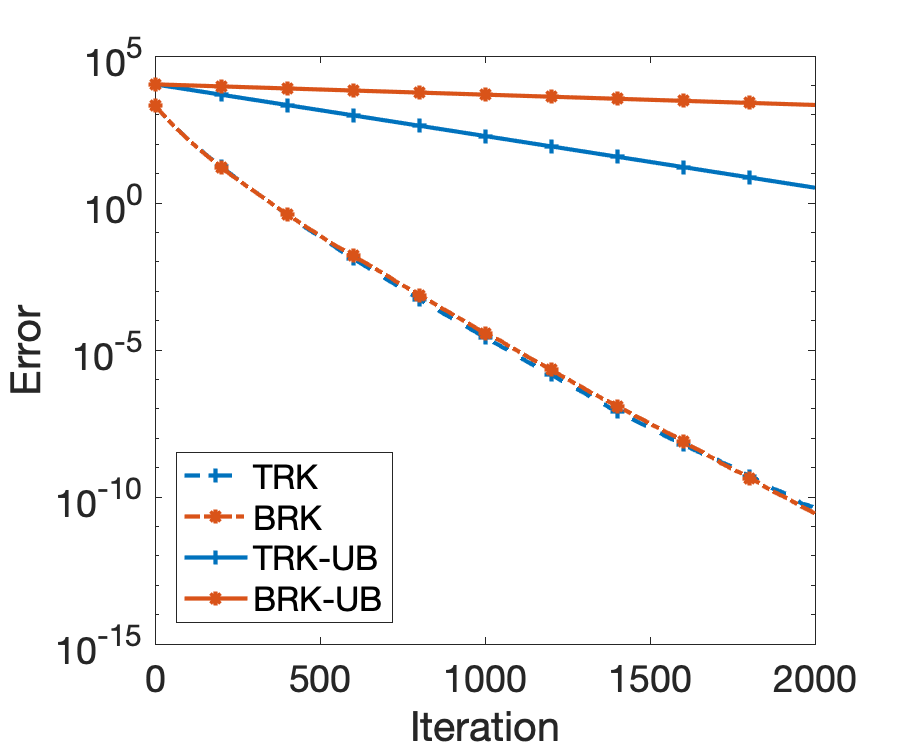}
    \caption{Performance of TRK and block MRK on a tensor linear system. `TRK-UB' and `BRK-UB' indicate the theoretical upper bounds of TRK and block MRK respectively.}
    \label{fig:exp_trk_brk}
\end{figure}

\section{Conclusion}
This work extends the randomized Kaczmarz literature to solve large-scale tensor linear systems under the t-product. The proposed tensor randomized Kaczmarz (TRK) algorithm solves large-scale tensor linear systems and is guaranteed to convergence exponentially in expectation. Connections to the block randomized Kaczmarz are made and empirical results are provided to support derived theoretical guarantees. This work further provides a framework to extend other stochastic iterative methods that arise in literature such as the randomized extended Kaczmarz algorithm, randomized Gauss-Seidel algorithm, coordinate descent, sketch-and-project \cite{gower2015randomized}, and many more.

\section*{Acknowledgements}
This work began at the 2019 workshop for Women in Science of Data and Math (WISDM) held at the Institute for Computational and Experimental Research in Mathematics (ICERM). This workshop is partially supported by an NSF ADVANCE grant (award \#1500481) to the Association for Women in Mathematics (AWM). Ma was partially supported by U.S. Air Force Award FA9550-18-1-0031 led by Roman Vershynin. Molitor is grateful to and was partially supported by NSF CAREER DMS $\#1348721$ and NSF BIGDATA DMS $\#1740325$ led by Deanna Needell. The authors would also like to thank Misha Kilmer for her advising during the WISDM workshop and valuable feedback that improved earlier versions of this manuscript. 

\appendix

\section{Proofs for properties of block circulant matrices}\label{sec:bcirc_proofs}
We use properties of circulant matrices and the Kronecker product in proving \Cref{fact:bcirc_distr,fact:bcirc_transpose}. 
For a vector $v\in\C^n$
\begin{equation*}
    \circulant{v} = 
    \begin{pmatrix}
    v_0 & v_{n-1} & \dots & v_1\\
    v_1 & v_{0} & \dots & v_2\\
    \vdots & \vdots & \ddots & \vdots\\
    v_{n-2} & v_{n-3} & \dots & v_{n-1}\\
    v_{n-1} & v_{n-2} & \dots & v_0\\
    \end{pmatrix}.
\end{equation*}

The block circulant of a matrix $\bcirc{\cM}$ can be decomposed as 
    \begin{equation}\label{eqn:circ_form}
        \bcirc{\cM} 
        = \sum_{i=0}^{n-1} \circulant{e_i} \otimes \cM_{::i},
    \end{equation}
 where $e_i$ is the $i\thup$ standard basis vector in $\C^n$ and $\otimes$ denotes the Kronecker product.

\subsection{Proof of \Cref{fact:bcirc_distr}}
Recall that \Cref{fact:bcirc_distr} states
    \begin{equation*}
        \bcirc{\cA\cB} = \bcirc{\cA}\bcirc{\cB}.
    \end{equation*}
\begin{proof}
Decomposing $\bcirc{\cA\cB}$, \Cref{eqn:circ_form} obtains the equality

\begin{equation*}
    \bcirc{\cA\cB} 
    = \sum_{i=0}^{n-1} \circulant{e_i} \otimes\left(\cA\cB\right)_{::i}.
\end{equation*}

For notational simplicity, let $\mA_i = \cA_{::i}$ and $\mB_i = \cB_{::i}$ denote the $i\thup$ frontal faces of $\cA$ and $\cB$ respectively. Then
    \begin{align*}
        \left(\cA\cB\right)_{::i} 
        &= \fold{\bcirc{\cA} \unfold{\cB}}_{::i}\\
        &= \begin{pmatrix}
        \mA_i & \mA_{i-1 } & \dots \mA_0 & \mA_{n-1} & \dots \mA_{i+1}
        \end{pmatrix}\unfold{\cB}\\
        &=
        \mA_i \mB_{0} +  \mA_{i-1 }\mB_{1} +  \dots \mA_0\mB_{i} + \mA_{n-1}\mB_{i+1} +  \dots \mA_{i+1}
       \mB_{n-1} \\
       &=\sum_{k=0}^{n-1} \mA_{i - k\pmod n} \mB_{k}.
    \end{align*}

    We then have that 
    \begin{equation*}
        \bcirc{\cA\cB} 
        = \sum_{i=0}^{n-1} \sum_{k=0}^{n-1}  \circulant{e_i} \otimes \mA_{i - k \pmod n} \mB_{k} .
    \end{equation*}
    Changing $i\to i+k$, we can rewrite this as 
    \begin{equation}\label{eqn:bcirc_AB_expanded}
        \bcirc{\cA\cB} 
        = \sum_{i=0}^{n-1} \sum_{k=0}^{n-1} \circulant{e_{i+k\pmod n}}  \otimes \mA_{i} \mB_{k} .
    \end{equation}
 
    Similarly, we can decompose $\bcirc{\cA}\bcirc{\cB}$ as 
    \begin{align*}
        \bcirc{\cA}\bcirc{\cB}
        &= \left(\sum_{i=0}^{n-1} \circulant{e_i} \otimes \mA_i  \right)
        \left(\sum_{k=0}^{n-1} \circulant{e_k} \otimes \mB_k  \right)\\
        &= \sum_{i=0}^{n-1} \sum_{k=0}^{n-1} \left(\circulant{e_i} \otimes \mA_i  \right)
        \left( \circulant{e_k} \otimes \mB_k \right).
    \end{align*}
    The mixed-product property further gives
    \begin{align*}
        \bcirc{\cA}\bcirc{\cB}
        &= \sum_{i=0}^{n-1} \sum_{k=0}^{n-1}  \circulant{e_i} \circulant{e_k} \otimes \mA_i\mB_k  \\
        &= \sum_{i=0}^{n-1} \sum_{k=0}^{n-1} \circulant{e_{i+k \pmod n}} \otimes \mA_i\mB_k  .
    \end{align*}
    We have now recovered the right-hand side of \Cref{eqn:bcirc_AB_expanded} and thus $\bcirc{\cA\cB}=
    \bcirc{\cA}\bcirc{\cB}$ as desired.
\end{proof}

\subsection{Proof of \Cref{fact:bcirc_transpose}}
\Cref{fact:bcirc_transpose} states
    \begin{equation*}
        \bcirc{\cM^*} = \bcirc{\cM}^*.
    \end{equation*}
\begin{proof}
For simplicity, let $\mM_i = \cM_{::i}$ denote the $i\thup$ frontal face of $\cM$. Decomposing $\bcirc{\cM}$ as in \Cref{eqn:circ_form}, using the definition of the tensor transpose, the fact that $\left(\mA \otimes \mB\right)^* = \mA^* \otimes \mB^*$ and $\circulant{e_i}^* = \circulant{e_{n-i}}$ \cite{kra2012circulant},
    \begin{align*}
        \bcirc{\cM^*} 
        &\overset{\eqref{eqn:circ_form}}{=} \sum_{i=0}^{n-1} \circulant{e_i} \otimes \left(\cM^*\right)_{::i} \\
        &= \mI_n \otimes \left(\mM_{1}\right)^* + \sum_{i=1}^{n-1}  \circulant{e_i} \otimes \left(\mM_{n-i}\right)^* \\
        &=  \mI_n^* \otimes \left(\mM_{1}  \right)^*+ \sum_{i=1}^{n-1}\circulant{e_{n-1}}^* \otimes \left(\mM_{n-i}  \right)^*\\
        &= \left[\mI_n  \otimes \left(\mM_{1} \right)+ \sum_{i=1}^{n-1} \circulant{e_{n-i}} \otimes \left(\mM_{n-i}  \right)\right]^*\\
        &= \bcirc{\cM}^*.
    \end{align*}
\end{proof}

\section{Tensor Pythagorean Theorem}
An analogue of the Pythagorean Theorem is stated and proved for tensors.
\begin{lemma}\label{lem:pythag}
    For an orthogonal projection $\cP$ and tensor $\cM$ of compatible size,
    \begin{equation*}
        \norm{\cM}_F^2 = \norm{\left(\cI - \cP\right)\cM}_F^2 +\norm{ \cP\cM}_F^2.
    \end{equation*}
\end{lemma}
\begin{proof}
This result can be shown by rewriting the tensor products in terms of matrix products and applying \Cref{lem:bcirc_proj}. Note that for a tensor $\cM$,
\begin{equation*}
    \norm{\cM}_F^2 = \norm{\unfold{\cM}}_F^2 .
\end{equation*}
Decomposing $\norm{\cM}_F^2$ and rewriting the result in terms of matrices,
\begin{align*}
    \norm{\cM}_F^2 
    &= \norm{\left(\cI - \cP\right)\cM + \cP\cM}_F^2 \\
    &= \norm{\bcirc{ \cI - \cP}\unfold{ \cM} +  \bcirc{\cP}\unfold{ \cM}}_F^2.
\end{align*}
Since $\bcirc{\cP}$ is an orthogonal projection (\Cref{lem:bcirc_proj}), by the Pythagorean theorem,
\begin{align*}
     \norm{\cM}_F^2 
    &= \norm{\bcirc{ \cI - \cP}\unfold{ \cM}}_F^2 + \norm{ \bcirc{\cP}\unfold{ \cM}}_F^2\\
    &\quad = \norm{\left(\cI - \cP\right)\cM}_F^2 +\norm{ \cP\cM}_F^2.
\end{align*}    
\end{proof}

\bibliographystyle{abbrv}
\bibliography{mybib}

\end{document}